\input ./style/arxiv-vmsta.cfg
\documentclass[numbers,compress,v1.0.1]{vmsta}
\usepackage{vtexbibtags}

\usepackage{bbm}
\usepackage{enumerate}
\usepackage[shortlabels]{enumitem}

\volume{6}% Updated by VTEXPTS2LaTeX.exe, 14.10.2019 12:34
\issue{3}% Updated by VTEXPTS2LaTeX.exe, 14.10.2019 12:34
\pubyear{2019}
\firstpage{269}% Updated by VTEXPTS2LaTeX.exe, 14.10.2019 12:34
\lastpage{284}% Updated by VTEXPTS2LaTeX.exe, 14.10.2019 12:34
\aid{VMSTA135}% Updated by VTEXPTS2LaTeX.exe, 10.06.2019 11:50
\doi{10.15559/19-VMSTA135}% Updated by VTEXPTS2LaTeX.exe, 10.06.2019
\articletype{research-article}

\startlocaldefs

\ifdefined\HCode % tex4ht
\def\index#1{}
\else % LaTeX
\fi
\allowdisplaybreaks

\theoremstyle{definition}
\newtheorem{definition}{Definition}

\newtheorem{remark}[definition]{Remark}
\newtheorem{example}[definition]{Example}
\newtheorem{examples}[definition]{Examples}
\newtheorem*{exampcont}{Example~\ref{example:vanishing} continued}

\theoremstyle{theorem}
\newtheorem{proposition}[definition]{Proposition}
\newtheorem{lemma}[definition]{Lemma}
\newtheorem{corollary}[definition]{Corollary}

\numberwithin{equation}{section}

\numberwithin{definition}{section}
\endlocaldefs

\begin{document}

\begin{frontmatter}
\pretitle{Research Article}

\title{Arithmetic of (independent) sigma-fields on probability spaces}

\author{\inits{M.}\fnms{Matija}~\snm{Vidmar}\ead[label=e1]{matija.vidmar@fmf.uni-lj.si}\orcid{0000-0002-1217-4054}}
\address{Department of Mathematics, \institution{University of Ljubljana}, \cny{Slovenia}}

%\thankstext[id=f1]{}

%\dedicated{}

%\markboth{Authors}{Title}
\markboth{M. Vidmar}{Arithmetic of (independent) sigma-fields on probability spaces}

\begin{abstract}
This note gathers what is known about, and provides some new results
concerning the operations of intersection, of ``generated
$\sigma$-field'', and of ``complementation'' for (independent) complete
$\sigma$-fields on probability spaces.
\end{abstract}

\received{\sday{19} \smonth{1} \syear{2019}}% Updated by
%VTEXPTS2LaTeX.exe, 10.06.2019 11:50
\revised{\sday{17} \smonth{5} \syear{2019}}% Updated by
%VTEXPTS2LaTeX.exe, 10.06.2019 11:50
\accepted{\sday{17} \smonth{5} \syear{2019}}% Updated by
%VTEXPTS2LaTeX.exe, 10.06.2019 11:50
\publishedonline{\sday{20} \smonth{6} \syear{2019}}
\end{frontmatter}

%s1 #&#
\section{Introduction}

Let $(\varOmega,\mathcal{M},\mathbb{P})$ be a probability space and let
$\varLambda$ be the collection of all complete sub-$\sigma$-fields of
$\mathcal{M}$. (We stress here that $\mathbb{P}$ need not itself be
complete to begin with. Complete just means containing $0_{\varLambda}:=
\mathbb{P}^{-1}(\{0,1\})$ -- the $\mathbb{P}$-trivial events of
$\mathcal{M}$.) $\sigma(\times\times)$ (resp. $\overline{\sigma}(
\times\times)$) is the smallest (resp. complete) $\sigma$-field on
$\varOmega$ containing or making measurable whatever stands in lieu of
$\times\times$. Then for $\{\mathcal{X},\mathcal{Y},\mathcal{Z}\}
\subset\varLambda$ set $\mathcal{X}\land\mathcal{Y}:=\mathcal{X}
\cap\mathcal{Y}$ and $\mathcal{X}\lor\mathcal{Y}:=\sigma(
\mathcal{X}\cup\mathcal{Y})$; write $\mathcal{X}\perp\!\!\!\!\perp
\mathcal{Y}$ if $\mathcal{X}$ and $\mathcal{Y}$ are independent, in
which case set $\mathcal{X}+\mathcal{Y}:=\mathcal{X}\lor\mathcal{Y}$;
finally, say $\mathcal{X}$ is complemented by $\mathcal{Y}$ in
$\mathcal{Z}$, or that $\mathcal{Y}$ is a complement of $\mathcal{X}$
in $\mathcal{Z}$, if $\mathcal{Z}=\mathcal{X}+\mathcal{Y}$.

We are interested in exposing the salient ``arithmetical rules'' of the
operations $\land$, $\lor$, and especially of $+$ and the notion of
a complement, delineating their scope through (counter)examples. Apart
from pure intellectual curiosity, the justification for the interest in
such matters --- that may seem a bit ``dry'' at first --- can be seen
as coming chiefly from the following
observations.
%sources.

\smallskip
\textbf{(1)} Even though the concepts involved are prima facie very
simple, the topic is not trivial and intuition can often mislead. The
following examples give already a flavor of this; in them, and in the
rest of this paper, equiprobable sign means a
$(\{-1,1\},2^{\{-1,1\}})$-valued random element $\xi$ with
$\mathbb{P}(\xi=1)=\mathbb{P}(\xi=-1)=1/2$.

%e1.1 #&#
%
\begin{example}[$\land$-$\lor$ distributivity may fail]
\label{example:dist}%
\leavevmode
\begin{enumerate}[(a)]%
\item
\label{dist:basic}
If $\xi_{1}$ and $\xi_{2}$ are independent equiprobable signs,\index{independent ! equiprobable signs} then
taking $\mathcal{X}=\overline{ \sigma}(\xi_{1})$, $\mathcal{Y}=
\overline{
\sigma}(\xi_{1}\xi_{2})$ and $\mathcal{Z}=\overline{\sigma}(\xi_{2})$,
the $\sigma$-fields $\mathcal{X},\mathcal{Y},\mathcal{Z}$ are pairwise
independent\index{pairwise independent} and $(\mathcal{X}\lor\mathcal{Z})\land(\mathcal{Y}
\lor\mathcal{Z})=\overline{\sigma}(\xi_{1},\xi_{2})$, while
$(\mathcal{X}\land\mathcal{Y})\lor\mathcal{Z}=0_{\varLambda}\lor
\mathcal{Z}=\overline{\sigma}(\xi_{2})$; so $(\mathcal{X}\lor
\mathcal{Z})\land(\mathcal{Y}\lor\mathcal{Z})\ne(\mathcal
{X}\land
\mathcal{Y})\lor\mathcal{Z}$. The same example also shows that one does
not in general have $(\mathcal{X}\land\mathcal{Z})\lor(\mathcal{Y}
\land\mathcal{Z})= (\mathcal{X}\lor\mathcal{Y})\land\mathcal{Z}$.
\item
\cite[Exercise/Warning~4.12]{williams}%
\label{example:warning}
Let
$\mathsf{Y}=(\mathsf{Y}_{n})_{n\in\mathbb{N}_{0}}$
be a sequence
of independent equiprobable signs.\index{independent ! equiprobable signs} For $n\in\mathbb{N}$ define
$\mathsf{X}_{n}:=\mathsf{Y}_{0}\cdots\mathsf{Y}_{n}$;
set
$\mathcal{Y}:=\overline{\sigma}(\mathsf{Y}_{1},\mathsf
{Y}_{2},\ldots
)$
and
$\mathcal{X}_{n}:=\overline{\sigma}(\mathsf{X}_{m}:m\in
\mathbb{N}_{\geq n})$ for $n\in\mathbb{N}$.
Then the $\mathcal{X}_{n}$, $n\in\mathbb{N}$, are decreasing, but
$\land_{n\in\mathbb{N}}(
\mathcal{X}_{n}\lor\mathcal{Y})\ne(\land_{n\in\mathbb{N}_{0}}
\mathcal{X}_{n})\lor\mathcal{Y}$.
Indeed the
$\mathsf{X}_{n}$,
$n\in\mathbb{N}$, are independent equiprobable signs,\index{independent ! equiprobable signs} so by
Kolmogorov's zero-one law $\land_{n\in\mathbb{N}}\mathcal{X}_{n}=0_{
\varLambda}$. On the other hand
$\mathsf{Y}_{0}$
is measurable w.r.t.
$\overline{\sigma}(\mathsf{Y})=\land_{n\in\mathbb{N}}(\mathcal{X}
_{n}\lor\mathcal{Y})$ and at the same time it is independent of
$\mathcal{Y}$. (For another related example see \cite{Yor1992}.)
\end{enumerate}
\end{example}

%e1.2 #&#
%
\begin{example}[Complements may not exist]
\label{example:not-exist}
If $\xi_{1}$, $\xi_{2}$ are independent equiprobable signs,\index{independent ! equiprobable signs} then
$\overline{\sigma}(\{\xi_{1}=1\}\cup\{\xi_{1}=-1,\xi_{2}=1\})$
has no
complement in $\overline{\sigma}(\xi_{1},\xi_{2})$.
\end{example}

%e1.3 #&#
%
\begin{example}[Complements may not be unique]
\label{example:not-unique}
Take again a pair of independent equiprobable signs $\xi_{1}$\index{independent ! equiprobable signs} and
$\xi_{2}$. Then $\overline{\sigma}(\xi_{1})+\overline{\sigma}(\xi
_{2})=\overline{\sigma}(\xi_{1},\xi_{2})$ but also $\overline
{\sigma
}(\xi_{1})+\overline{\sigma}(\xi_{1}\xi_{2})=\overline{\sigma
}(\xi
_{1},\xi_{2})$.
\end{example}

%e1.4 #&#
%
\begin{example}[Vanishing of information in the limit]
\label{example:vanishing}%
\cite[Example~1.1; see also the references there]{tsirelson}. Let
$\varOmega=\{-1,1\}^{\mathbb{N}}$, and let $\xi_{i}$, $i\in\mathbb{N}$,
the canonical projections, be independent equiprobable signs\index{independent ! equiprobable signs} generating
$\mathcal{M}=(2^{\{-1,1\}})^{\otimes\mathbb{N}}$. Let $\mathcal{G}
_{n}=\overline{\sigma}(\xi_{1}\xi_{2},\ldots,\xi_{n}\xi_{n+1})$ and
$\mathcal{F}_{n}=\overline{\sigma}(\xi_{n+1},\xi_{n+2},\ldots)$ for
$n\in\mathbb{N}$. Then $\mathcal{G}_{n}+\mathcal{F}_{n}=\mathcal{F}
_{0}=\mathcal{M}$ for all $n\in\mathbb{N}$, and by Kolmogorov's
zero-one law $\mathcal{F}_{\infty}:=\land_{n\in\mathbb{N}}
\mathcal{F}_{n}=0_{\varLambda}$. Furthermore, we have $\mathcal{F}_{n}=
\mathcal{F}_{n+1}+ \mathcal{H}_{n+1}$ and $\mathcal{G}_{n+1}=
\mathcal{G}_{n}+\mathcal{H}_{n+1}$ for all $n\in\mathbb{N}_{0}$, if we
put $\mathcal{H}_{n}:=\mathcal{G}_{n}\land\mathcal
{F}_{n-1}=\overline{
\sigma}(\xi_{n}\xi_{n+1})$ for $n\in\mathbb{N}$. But still
$\mathcal{G}_{\infty}:=\lor_{n\in\mathbb{N}}\mathcal
{G}_{n}=\overline{
\sigma}(\xi_{1}\xi_{2},\xi_{2}\xi_{3},\ldots)\ne\mathcal{M}$, for
instance, because $\xi_{1}$ is non-trivial and independent of~$\mathcal{G}_{\infty}$.
\end{example}
Concerning the failure of the equality $\land_{n\in\mathbb{N}}(
\mathcal{X}_{n}\lor\mathcal{Y})= (\land_{n\in\mathbb{N}_{0}}
\mathcal{X}_{n})\lor\mathcal{Y}$ in
Example~\ref{example:dist}\ref{example:warning}, Chaumont and Yor
\cite[p.~30]{chaumont-yor} write: ``A number of authors, (including the
present authors, separately!!), gave wrong proofs of /this equality/
under various hypotheses. This seems to be one of the worst traps
involving $\sigma$-fields.'' According to Williams
\cite[p. 48]{williams}: ``The phenomenon illustrated by this example
tripped up even Kolmogorov and Wiener. [...] Deciding when we
\emph{can} assert [equality] is a tantalizing problem in many
probabilistic contexts.'' {\'{E}}mery and Schachermayer
\cite[p. 291]{emery} call a variant of Example~\ref{example:vanishing}
``paradigmatic [...], well-known in ergodic theory, [...], independently
discovered by several authors''.

\smallskip
\textbf{(2)} In spite of the subtleties involved, facts concerning the
arithmetic of $\sigma$-fields are not very easily accessible in the
literature, various partial results being scattered across papers and
monographs, as and when the need for them arose.

\smallskip
\textbf{(3)} In broad sense, nondecreasing families of sub-$\sigma$-fields
--- filtrations --- model the flow of information in a probabilistic
context. They are essential to the modern-day proper understanding of
martingales and Markov processes. And since stochastic models are
usually specified by some kind of (conditional) independence structure\index{independence}
(think i.i.d. sequences, L\'{e}vy processes, Markov processes in
general), it is therefore important to understand how such information,
as embodied by $\sigma$-fields, is ``aggregated'' and/or
``intersected'' over (conditionally) independent $\sigma$-fields. The
classical increasing and decreasing martingale convergence theorems\index{martingale convergence}
\cite[Theorem~6.23]{kallenberg}, for instance, involve the generated and
intersected $\sigma$-fields in a key way. Kolmogorov's zero-one law and
its extensions \cite[Corollary 6.25]{kallenberg}, with their many
offsprings, are another example in which the interplay between
independence,\index{independence} intersected, and generated $\sigma$-fields lies at the
very heart of the matter.

\smallskip
\textbf{(4)} More narrowly, the exposition in \cite{tsirelson}
recognizes stochastic noises (generalizations of Wiener and Poissonian
noise) as subsets of the lattice $\varLambda$ satisfying in particular,
and in an essential way, a certain property with respect to independent
complements;\index{independent ! complements} see also \cite{feldman,schramm}.

\smallskip
With the above as motivation,
%then,
and following the introduction of
some further notation and preliminaries in
Section~\ref{section:notation}, we investigate below in
Section~\ref{section:arithmetic}, in depth: (I) the distributivity
properties\index{distributivity} of the pair $\land$-$\lor$ for families of $\sigma$-fields
that, roughly speaking, exhibit at least some independence properties\index{independence}
between them;
(II)~the properties of complements\index{complements}
(existence, uniqueness,
etc.). In particular, apart from some trivial observations, we confine
our attention to those statements concerning the arithmetic of
$\sigma$-fields, in which a property of (conditional) independence
intervenes\index{independence} in a non-trivial way (this is of course automatic for (II));
hence the title. For the most part the paper is of an expository nature;
see below for the precise references. In some places a couple of
original complements\index{complements}/extensions are provided.
Section~\ref{section:application} closes with a brief application; other
uses of the presented results are
%to be
found in the citations that we
%shall make (have made),
include, as well as in the literature quoted in those.

%s2 #&#
\section{Further notation and preliminaries}%
\label{section:notation}
Some general notation and vocabulary. For $M\subset [-\infty,\infty]$, $\mathcal{B}_{M}$ will denote the
Borel $\sigma$-field on $M$ for the standard (Euclidean) topology thereon. For
$\sigma$-fields $\mathcal{F}$ and $\mathcal{G}$, $\mathcal{F}/
\mathcal{G}$ is the set of precisely all the
$\mathcal{F}/\mathcal{G}$-measurable maps. A measure on a
$\sigma$-field that contains the singletons of the underlying space
will be said to be diffuse,\index{diffuse} or continuous, if it does not charge any
singleton. Throughout ``a.s.'' is short for ``$\mathbb{P}$-almost
surely'' and $\mathbb{E}$ denotes expectation w.r.t. $\mathbb{P}$. A
random element valued in $([0,1],\mathcal{B}_{[0,1]})$ whose law is the
(trace of) Lebesgue measure on $[0,1]$ will be said to have (the)
uniform law (on $[0,1]$).
\label{notation}

Let now $\{\mathcal{X},\mathcal{Y}\}\subset\varLambda$. Then (i) for
$\mathsf{M}\in\mathcal{M}/\mathcal{B}_{[-\infty,\infty]}$,
$\mathbb{E}[\mathsf{M}\vert\mathcal{X}]$ is the conditional expectation
of $\mathsf{M}$ w.r.t. $\mathcal{X}$ (when $\mathbb{E}[\mathsf{M}^{+}]
\land\mathbb{E}[\mathsf{M}^{-}]<\infty$, in which case $\mathbb{E}[
\mathsf{M}\vert\mathcal{X}]\in\mathcal{X}/\mathcal{B}_{[-\infty,
\infty]}$)\footnote{We will indulge in the usual confusion between
measurable functions and their equivalence classes mod $\mathbb{P}$.
Because we will only be interested in complete $\sigma$-fields this
will be of no consequence.} and as usual $\mathbb{P}[F\vert
\mathcal{X}]=\mathbb{E}[\mathbbm{1}_{F}\vert\mathcal{X}]$ for
$F\in\mathcal{M}$; (ii) we will denote by $\mathbb{E}_{\vert
\mathcal{X}}$ the operator, on $L^{1}(\mathbb{P})$, of the conditional
expectation w.r.t. $\mathcal{X}$: so $\mathbb{E}_{\vert\mathcal{X}}(
\mathsf{M})=\mathbb{E}[\mathsf{M}\vert\mathcal{X}]$ a.s. for
$\mathsf{M}\in L^{1}(\mathbb{P})$; (iii) $\mathcal{X}$ will be said to
be countably generated up to negligible sets, or to be essentially
separable, if there is a denumerable $\mathcal{B}\subset\mathcal{X}$
such that $\mathcal{X}=\overline{\sigma}(\mathcal{B})$: manifestly it
is so if and only if $L^{1}(\mathbb{P}\vert_{\mathcal{X}})$ is
separable, in which case every element $\mathcal{Y}\in\varLambda$ with
$\mathcal{Y}\subset\mathcal{X}$ is countably generated up to negligible
sets, and this is true if and only if there is an $\mathsf{X}\in
\mathcal{X}/\mathcal{B}_{\mathbb{R}}$ with $\mathcal{X}=\overline{
\sigma}(\mathsf{X})$; (iv) if further $\mathcal{Z}\in\varLambda$, we
will write $\mathcal{X}\perp\!\!\!\!\perp_{\mathcal{Z}} \mathcal{Y}$
to mean that $\mathcal{X}$ and $\mathcal{Y}$ are independent given
$\mathcal{Z}$.

%r2.1 #&#
%
\begin{remark}
A warning: separability per se is not hereditary. For instance
$\mathcal{B}_{\mathbb{R}}$ is countably generated but the
countable-co-countable $\sigma$-field on $\mathbb{R}$ is not. In
general it is true that completeness will have a major role to play in
what follows, and we shall make no apologies for restricting our
attention to complete sub-$\sigma$-fields from the get go --
practically none of the results presented would be true without this
assumption (or would be true only ``mod $\mathbb{P}$'', which amounts
to the same thing).
\end{remark}

The following basic facts about conditional expectations are often
useful; we will use them silently throughout.
%
%l2.2 #&#
%
\begin{lemma}[Independent conditioning\index{independent ! conditioning}]
\label{lema}
Let $\{\mathsf{F},\mathsf{G}\}\subset\mathcal{M}/\mathcal{B}_{[0,
\infty]}$ and let $\{\mathcal{X},\mathcal{Y},\break \mathcal{Z}\}\subset
\varLambda$. If $\mathcal{Y}\lor\sigma(\mathsf{G})\perp\!\!\!\!
\perp\mathcal{X}\lor\sigma(\mathsf{F})$, then $\mathbb{E}[
\mathsf{F}\mathsf{G}\vert\mathcal{X}\lor\mathcal{Y}]=\mathbb{E}[
\mathsf{F}\vert\mathcal{X}]\mathbb{E}[\mathsf{G}\vert\mathcal{Y}]$
a.s.; in particular if $\mathcal{Y}\perp\!\!\!\!\perp\mathcal{X}
\lor\mathcal{Z}$, then $\mathcal{Z}\perp\!\!\!\!\perp_{\mathcal{X}}
\mathcal{Y}$; finally, if $\sigma(\mathsf{F})\perp\!\!\!\!
\perp_{\mathcal{X}}\mathcal{Y}$, then $\mathbb{E}[\mathsf{F}\vert
\mathcal{X}\lor\mathcal{Y}]=\mathbb{E}[\mathsf{F}\vert\mathcal{X}]$
a.s.
\end{lemma}
\begin{proof}
For the first claim, by a $\pi/\lambda$-argument it suffices to check
that $\mathbb{E}[\mathsf{F}\mathsf{G};X\cap Y]=\mathbb{E}[\mathbb{E}[
\mathsf{F}\vert\mathcal{X}] \mathbb{E}[\mathsf{G}\vert\mathcal{Y}];X
\cap Y]$ for $X\in\mathcal{X}$ and $Y\in\mathcal{Y}$, which is
immediate (both sides are equal to $\mathbb{E}[\mathsf{F};X]
\mathbb{E}[\mathsf{G};Y]$ on account of $\mathcal{Y}\lor\sigma(
\mathsf{G})\perp\!\!\!\!\perp\mathcal{X}\lor\sigma(\mathsf{F})$).
To obtain the second statement, let $\mathsf{Z}\in\mathcal{Z}/
\mathcal{B}_{[0,\infty]}$ and $\mathsf{Y}\in\mathcal{Y}/\mathcal{B}
_{[0,\infty]}$; then \xch{a.s.}{a,s.} $\mathbb{E}[\mathsf{Z}\mathsf{Y}\vert
\mathcal{X}]=\mathbb{E}[\mathsf{Z}\mathsf{Y}\vert\mathcal{X}\lor0_{
\varLambda}]=\mathbb{E}[\mathsf{Z}\vert\mathcal{X}]\mathbb{E}[
\mathsf{Y}]=\mathbb{E}[\mathsf{Z}\vert\mathcal{X}]\mathbb{E}[
\mathsf{Y}\vert\mathcal{X}]$. For the final claim, by a $\pi
/\lambda
$-argument it suffices to check that $\mathbb{E}[\mathsf{F};X\cap Y]=
\mathbb{E}[\mathbb{E}[\mathsf{F}\vert\mathcal{X}];X\cap Y]$ for all
$(X,Y)\in\mathcal{X}\times\mathcal{Y}$. But $\mathbb{E}[\mathbb{E}[
\mathsf{F}\vert\mathcal{X}];X\cap Y]=\mathbb{E}[\mathbb{E}[
\mathsf{F}\vert\mathcal{X}]\mathbb{P}[Y\vert\mathcal{X}];X]=
\mathbb{E}[\mathbb{E}[\mathsf{F}\mathbbm{1}_{Y}\vert\mathcal{X}];X]$,
which is indeed equal to $\mathbb{E}[\mathsf{F};X\cap Y]$.
\end{proof}
We conclude this section with a statement concerning decreasing
convergence for martingales indexed by a directed set (it is also true
in its increasing convergence guise
\cite[Proposition~V-1-2]{neveu} but we shall not find use of that
version). In it, and in the remainder of this paper, for a family
$(\mathcal{X}_{t})_{t\in T}$ in $\varLambda$ we set $\land_{t\in T}
\mathcal{X}_{t}:=\cap_{t\in T} \mathcal{X}_{t}$, provided $T$ is
non-empty (similarly, later on, we will use the notation $\lor_{t
\in T}\mathcal{X}_{t}:=\overline{\sigma}(\cup_{t\in T}\mathcal{X}
_{t})$ ($=0_{\varLambda}$ when $T$ is empty)).

%l2.3 #&#
%
\begin{lemma}[Decreasing martingale convergence\index{martingale convergence}]
Let $\mathsf{X}\in L^{1}(\mathbb{P})$ and let $(\mathcal{X}_{t})_{t
\in T}$ be a non-empty net in $\varLambda$ indexed by a directed set
$(T,\leq)$ satisfying $\mathcal{X}_{t}\subset\mathcal{X}_{s}$ whenever
$s\leq t$ are from $T$. Then the net $(\mathbb{E}[\mathsf{X}\vert
\mathcal{X}_{t}])_{t\in T}$ converges in $L^{1}(\mathbb{P})$ to
$\mathbb{E}[\mathsf{X}\vert\land_{t\in T}\mathcal{X}_{t}]$.
\end{lemma}
%
%r2.4 #&#
%
\begin{remark}
Recall that when $T=\mathbb{N}$ with the usual order, then the
convergence is also almost sure.
\end{remark}
\begin{proof}
According to \cite[Lemma~V-1-1]{neveu} and the usual decreasing
martingale convergence\index{martingale convergence} indexed by $\mathbb{N}$
\cite[Corollary~V-3-12]{neveu} the net $(\mathbb{E}[\mathsf{X}\vert
\mathcal{X}_{t}])_{t\in T}$ is convergent to some $\mathsf{X}_{\infty
}$ in $L^{1}(\mathbb{P})$. Because for each $t\in T$, $L^{1}(
\mathbb{P}\vert_{\mathcal{X}_{t}})$ is closed in $L^{1}(\mathbb
{P})$ and
since $\mathsf{X}_{\infty}$ is also the limit of the net $(
\mathbb{E}[\mathsf{X}\vert\mathcal{X}_{u}])_{u\in T_{\geq t}}$, it
follows that $\mathsf{X}_{\infty}\in\mathcal{X}_{t}/\mathcal{B}_{
\mathbb{R}}$; hence $\mathsf{X}_{\infty}\in(\land_{t\in T}
\mathcal{X}_{t})/\mathcal{B}_{\mathbb{R}}$. Then for any $X\in
\land_{t\in T}\mathcal{X}_{t}$, $\mathbb{E}[\mathsf{X}_{\infty}; X]=
\lim_{t\in T} \mathbb{E}[\mathbb{E}[\mathsf{X}\vert\mathcal{X}_{t}];X]=
\lim_{t\in T} \mathbb{E}[\mathsf{X};X]=\mathbb{E}[\mathsf{X};X]$, which
means that a.s. $\mathsf{X}_{\infty}=\mathbb{E}[\mathsf{X}\vert
\land_{t\in T}\mathcal{X}_{t}]$.
\end{proof}

%s3 #&#
\section{The arithmetic}%
\label{section:arithmetic}
We begin with some simple observations.

%r3.1 #&#
%
\begin{remark}[Lattice structure]
\cite[passim]{tsirelson}. The operations $\land$, $\lor$ in
$\varLambda$ are clearly associative and commutative, and one has the
absorption laws: $(\mathcal{X}\land\mathcal{Y})\lor\mathcal{X}=
\mathcal{X}$ and $(\mathcal{X}\lor\mathcal{Y})\land\mathcal{X}=
\mathcal{X}$ for $\{\mathcal{X},\mathcal{Y}\}\subset\varLambda$. Besides,
$0_{\varLambda}\lor\mathcal{X}=\mathcal{X}$ and $\mathcal{X}\land
\mathcal{M}=\mathcal{X}$ for all $\mathcal{X}\in\varLambda$. Thus
$(\varLambda,\land,\lor)$ is a bounded algebraic lattice with bottom
$0_{\varLambda}$ and top $\mathcal{M}$. However, it is not distributive
in general, as we saw in the introduction. While $+$ is not an internal
operation on $\varLambda$, nevertheless we may assert, for $\{
\mathcal{X},\mathcal{Y},\mathcal{Z}\}\subset\varLambda$, that
$\mathcal{X}+\mathcal{Y}=\mathcal{Y}+\mathcal{X}$, resp. $(
\mathcal{X}+\mathcal{Y})+\mathcal{Z}=\mathcal{X}+(\mathcal{Y}+
\mathcal{Z})$, whenever $\mathcal{X}$ and $\mathcal{Y}$ are independent,
resp. and independent of $\mathcal{Z}$. Clearly also $\mathcal{X}+0_{
\varLambda}=\mathcal{X}$ for $\mathcal{X}\in\varLambda$.
\end{remark}

%p3.2 #&#
%
\begin{proposition}[Independence\index{independence} and commutativity]
\label{proposition:basic}%
\textup{\cite[Proposition~3.5]{tsirelson}}. Let $\{\mathcal{X},\break \mathcal{Y}
\}\subset\varLambda$. Then the following are equivalent.
\begin{enumerate}[(i)]%
\item
\label{basic:i}
$\mathcal{X}$ and $\mathcal{Y}$ are independent.
\item
\label{basic:ii}
$\mathcal{X}\land\mathcal{Y}=0_{\varLambda}$ and $\mathcal{X}$ and
$\mathcal{Y}$ ``commute'': $\mathbb{E}_{\vert\mathcal{X}} \mathbb{E}
_{\vert\mathcal{Y}}=\mathbb{E}_{\vert\mathcal{Y}} \mathbb{E}_{
\vert\mathcal{X}}$.
\item
\label{basic:iii}
$\mathbb{E}_{\vert\mathcal{X}} \mathbb{E}_{\vert\mathcal{Y}}=
\mathbb{E}_{\vert0_{\varLambda}}$.
\end{enumerate}
\end{proposition}
%
%e3.3 #&#
%
\begin{example}
\label{ex:commute}
Let $\xi_{1}$, $\xi_{2}$ be independent equiprobable signs\index{independent ! equiprobable signs} and
$\mathcal{X}=\overline{\sigma}(\{\xi_{1}=\xi_{2}=1\})$, $\mathcal
{Y}=\overline{
\sigma}(\xi_{1})$. Then $\mathcal{X}$ and $\mathcal{Y}$ are not
independent but $\mathcal{X}\land\mathcal{Y}=0_{\varLambda}$.
\end{example}
\begin{proof}
\ref{basic:ii} implies \ref{basic:iii} because $\mathbb{E}_{\vert
\mathcal{X}}\mathbb{E}_{\vert\mathcal{Y}}=\mathbb{E}_{\vert
\mathcal{Y}} \mathbb{E}_{\vert\mathcal{X}}$ entails that $\mathbb{E}
_{\vert\mathcal{X}} \mathbb{E}_{\vert\mathcal{Y}}=\mathbb{E}_{
\vert\mathcal{Y}} \mathbb{E}_{\vert\mathcal{X}}=\mathbb{E}_{\vert
\mathcal{X}\land\mathcal{Y}}$. Also, if $\mathcal{X}$ and $
\mathcal{Y}$ are independent, then the basic properties of conditional
expectations imply $\mathbb{E}_{\vert\mathcal{X}} \mathbb{E}_{\vert
\mathcal{Y}}=\mathbb{E}_{\vert0_{\varLambda}}=\mathbb{E}_{\vert
\mathcal{Y}} \mathbb{E}_{\vert\mathcal{X}}$, while clearly
$\mathcal{X}\land\mathcal{Y}=0_{\varLambda}$, i.e. \ref{basic:i} implies
\ref{basic:ii}. Suppose now \ref{basic:iii}. Let $X\in\mathcal{X}$ and
$Y\in\mathcal{Y}$. Then $\mathbb{P}(X\cap Y)=\mathbb{E}[\mathbb{P}[Y
\vert\mathcal{X}];X]=\mathbb{E}[\mathbb{E}[\mathbbm{1}_{Y}\vert
\mathcal{Y}\vert\mathcal{X}];X]=\mathbb{E}[\mathbb{P}[Y\vert0_{
\varLambda}];X]=\mathbb{P}(X)\mathbb{P}(Y)$, which is \ref{basic:i}.
\end{proof}
The next few results deal with the distributivity properties\index{distributivity} of the pair
$\lor$-$\land$, when there are strong independence properties.
%
%p3.4 #&#
%
\begin{proposition}[Distributivity\index{distributivity} I]
\label{distributivity}
Let $(\mathcal{X}_{\alpha\beta})_{(\alpha,\beta)\in\mathfrak{A}
\times\mathfrak{B}}$ be a family in $\varLambda$, $\mathfrak{A}$
non-empty, such that the $\mathcal{Z}_{\beta}:=
\lor_{\alpha\in\mathfrak{A}}\mathcal{X}_{\alpha\beta}$,
$\beta\in\mathfrak{B}$, are independent. Then
%
%e3.1 #&#
%
\begin{equation}
\label{eq:dist} \land_{\alpha\in\mathfrak{A}}\lor_{\beta\in\mathfrak{B}}
\mathcal{X}_{\alpha\beta}=\lor_{\beta\in\mathfrak{B}} \land_{\alpha\in\mathfrak{A}}
\mathcal{X}_{\alpha\beta}.
\end{equation}
\end{proposition}
It is quite agreeable that the preceding statement can be made in such
generality. We give some remarks before turning to its proof.
%
%r3.5 #&#
%
\begin{remark}
Of course the independence\index{independence} of $\mathcal{Z}_{\beta}$, $\beta\in
\mathfrak{B}$, is far from being necessary in order for
\eqref{eq:dist} to prevail. For instance if $\{\mathcal{X},
\mathcal{Y},\mathcal{Z}\}\subset\varLambda$, and $\mathcal{Z}\subset
\mathcal{X}$ or $\mathcal{Z}\subset\mathcal{Y}$, then $(\mathcal{X}
\land\mathcal{Z})\lor(\mathcal{Y}\land\mathcal{Z})=\mathcal{Z}=(
\mathcal{X}\lor\mathcal{Y})\land\mathcal{Z}=(\mathcal{X}\lor
\mathcal{Y})\land(\mathcal{Z}\lor\mathcal{Z})$, but $\mathcal{X}
\lor\mathcal{Z}$ and $\mathcal{Y}\lor\mathcal{Z}$ are not independent
unless $\mathcal{Z}=0_{\varLambda}$; similarly if $\mathcal{X}\lor
\mathcal{Y}\subset\mathcal{Z}$, then $(\mathcal{X}\land\mathcal{Y})
\lor(\mathcal{Z}\land\mathcal{Z})=(\mathcal{X}\land\mathcal{Y})
\lor\mathcal{Z}=\mathcal{Z}=(\mathcal{X}\lor\mathcal{Z})\land(
\mathcal{Y}\lor\mathcal{Z})$, but $\mathcal{X}\lor\mathcal{Y}$ and
$\mathcal{Z}$ are not independent unless $\mathcal{X}=\mathcal{Y}=0_{
\varLambda}$.
\end{remark}
%
%r3.6 #&#
%
\begin{remark}
The generality of a not necessarily denumerable $\mathfrak{B}$ in
Proposition~\ref{distributivity} is~of only superficial value. Indeed
clearly we \xch{have}{hvae} $\lor_{\beta\in\mathfrak{B}}
\land_{\alpha\in\mathfrak{A}}\mathcal{X}_{\alpha\beta}=\break
\cup_{B\text{ countable }\subset\mathfrak{B}}\lor_{\beta\in B}
\land_{\alpha\in\mathfrak{A}}\mathcal{X}_{\alpha\beta}$; similarly
if $A\in\land_{\alpha\in\mathfrak{A}}\lor_{\beta\in\mathfrak{B}}
\mathcal{X}_{\alpha\beta}$, then for sure $A\in\lor_{\beta\in B}
\mathcal{Z}_{\beta}$ for some denumerable $B\subset\mathfrak{B}$ so
that, by the very statement of this proposition (with $\mathfrak{B}$ a
two-point set), $A\in\land_{\alpha\in\mathfrak{A}}\lor_{\beta
\in B}\mathcal{X}_{\alpha\beta}$, viz. $\land_{\alpha\in
\mathfrak{A}}\lor_{\beta\in\mathfrak{B}}\mathcal{X}_{\alpha\beta}=
\cup_{B\text{ countable }\subset\mathfrak{B}}
\land_{\alpha\in\mathfrak{A}}\lor_{\beta\in B}\mathcal
{X}_{\alpha
\beta}$.
\end{remark}

%r3.7 #&#
%
\begin{remark}
Proposition~\ref{distributivity} yields at once Kolmogorov's zero-one
law: if $\mathcal{A}=(\mathcal{A}_{\gamma})_{\gamma\in\varGamma}$ is
an independency (i.e. a family consisting of independent $\sigma$-fields) from $\varLambda$, independent from a $\mathcal{B}
\in\varLambda$ then, setting for cofinite $A\subset\varGamma$,
$\lor_{A} \mathcal{A}:=\lor_{\gamma\in A}\mathcal{A}_{\gamma}$, one
obtains $\land_{A\text{ cofinite in }\varGamma}(\mathcal{B}\lor(\lor
_{A}\mathcal{A}))=\mathcal{B}$.
\end{remark}
\begin{proof}
The inclusion $\supset$ in \eqref{eq:dist} is trivial. On the other
hand, for $\beta\in\mathfrak{B}$, $\land_{\alpha\in\mathfrak{A}}
\lor_{\beta'\in\mathfrak{B}}\mathcal{X}_{\alpha\beta'}\,{\subset}\,
\land_{\alpha\in\mathfrak{A}}(\mathcal{X}_{\alpha\beta}\lor(
\lor_{\beta'\in\mathfrak{B}\backslash\{\beta\}}\mathcal{Z}_{
\beta'})) $. Hence $\land_{\alpha\in\mathfrak{A}}
\lor_{\beta'\in\mathfrak{B}}\mathcal{X}_{\alpha\beta'}\,{\subset}\,
\land_{\beta\in\mathfrak{B}} (\land_{\alpha\in\mathfrak{A}}(
\mathcal{X}_{\alpha\beta}\lor(
\lor_{\beta'\in\mathfrak{B}\backslash\{\beta\}}\mathcal{Z}_{
\beta'})) )$, and thus it will suffice to prove
\eqref{eq:dist} for the following two special cases.
\begin{enumerate}[(a)]%
\item
\label{dist:a}
$\mathfrak{B}=\{1,2\}$, $\mathcal{X}_{\alpha2}=\mathcal{Z}_{2}$ for
$ \alpha\in\mathfrak{A}$.
\item
\label{dist:b}
$\mathfrak{A}=\mathfrak{B}$ and $\mathcal{X}_{\alpha\beta}=
\mathcal{Z}_{\beta}$ for $\alpha\ne\beta$ from $\mathfrak{A}$.
\end{enumerate}
In proving this we will use without special mention the completeness of
the members of $\varLambda$.

\ref{dist:a}. Relabel $\mathcal{X}_{\alpha1}=:\mathcal{X}_{\alpha}$,
$\alpha\in\mathfrak{A}$, and $\mathcal{Z}_{2}=:\mathcal{Y}$. Suppose
\eqref{eq:dist} has been established for $\mathfrak{A}$ finite (all the
time assuming \ref{dist:a}). Let $T$ consist of the finite non-empty
subsets of $\mathfrak{A}$, direct $T$ by inclusion $ \subset$, and
define $\underline{\mathcal{X}}_{A}:=\land_{\alpha\in A}\mathcal{X}
_{\alpha}$ for $A\in T$. Then $\land_{\alpha\in\mathfrak{A}}(
\mathcal{X}_{\alpha}\lor\mathcal{Y})=\land_{A\in T}( \underline{
\mathcal{X}}_{A}\lor\mathcal{Y})$ and (of course) $
\land_{\alpha\in\mathfrak{A}}\mathcal{X}_{\alpha}=\land_{A\in
T}\underline{
\mathcal{X}}_{A}$. Let $X\in\lor_{\alpha\in\mathfrak{A}}
\mathcal{X}_{\alpha}=:\mathcal{X}$ and $Y\in\mathcal{Y}$. Using
$\mathcal{X}\perp\!\!\!\!\perp\mathcal{Y}$ and decreasing martingale
convergence\index{martingale convergence} we see that a.s. $\mathbb{P}[X\cap Y\vert(\land_{A\in T}
\underline{\mathcal{X}}_{A})\lor\mathcal{Y}]=\mathbb{P}[X \vert
\land_{A\in T}\underline{\mathcal{X}}_{A}]\mathbb{P}[Y\vert
\mathcal{Y}]=(\lim_{A\in T}\mathbb{P}[X\vert\underline{\mathcal{X}}
_{A}])\mathbb{P}[Y\vert\mathcal{Y}]=\lim_{A\in T}(\mathbb
{P}[X\vert
\underline{\mathcal{X}}_{A}]\mathbb{P}[Y\vert\mathcal{Y}])=\lim_{A
\in T}\mathbb{P}[X\cap Y\vert\underline{\mathcal{X}}_{A}\lor
\mathcal{Y}]=\mathbb{P}[X\cap Y\vert\land_{A\in T}( \underline{
\mathcal{X}}_{A}\lor\mathcal{Y})]$, where the limits are in
$L^{1}(\mathbb{P})$. A $\pi/\lambda$-argument allows to conclude that
\eqref{eq:dist} holds true. Suppose now $\mathfrak{A}$ is finite. By
induction we may and do consider only the case $\mathfrak{A}=\{1,2\}$,
and so we are to show that $(\mathcal{X}_{1}\lor\mathcal{Y})\land(
\mathcal{X}_{2}\lor\mathcal{Y})=(\mathcal{X}_{1}\land\mathcal{X}
_{2})\lor\mathcal{Y}$. Let again $X\in\mathcal{X}$ and $Y\in
\mathcal{Y}$. Then using $\mathcal{X}\perp\!\!\!\!\perp\mathcal{Y}$,
convergence of iterated conditional expectations
\cite[Proposition~3]{burkholder} and bounded convergence, we obtain that
a.s. $\mathbb{P}[X\cap Y\vert(\mathcal{X}_{1}\lor\mathcal{Y})\land
(\mathcal{X}_{2}\lor\mathcal{Y})]=\mathbb{E}[\mathbbm{1}_{X\cap Y}
\vert\mathcal{X}_{1}\lor\mathcal{Y}\vert(\mathcal{X}_{1}\lor
\mathcal{Y})\land(\mathcal{X}_{2}\lor\mathcal{Y})]=\mathbb{E}[
\mathbb{P}[X\vert\mathcal{X}_{1}]\mathbbm{1}_{Y}\vert(\mathcal{X}
_{1}\lor\mathcal{Y})\land(\mathcal{X}_{2}\lor\mathcal{Y})]=
\mathbb{E}[\mathbb{P}[X\vert\mathcal{X}_{1}]\mathbbm{1}_{Y}\vert
\mathcal{X}_{2}\lor\mathcal{Y}\vert(\mathcal{X}_{1}\lor\mathcal{Y})
\land(\mathcal{X}_{2}\lor\mathcal{Y})]=\mathbb{E}[\mathbb{E}[
\mathbbm{1}_{X}\vert\mathcal{X}_{1}\vert\mathcal{X}_{2}]\mathbbm{1}
_{Y}\vert(\mathcal{X}_{1}\lor\mathcal{Y})\land(\mathcal{X}_{2}
\lor\mathcal{Y})]=\mathbb{E}[\mathbb{E}[\mathbbm{1}_{X}\vert
\mathcal{X}_{1}\vert\mathcal{X}_{2}\vert\mathcal{X}_{1}\vert
\mathcal{X}_{2}]\mathbbm{1}_{Y}\vert(\mathcal{X}_{1}\lor\mathcal{Y})
\land(\mathcal{X}_{2}\lor\mathcal{Y})]=\cdots\to\mathbb{E}[
\mathbb{P}[X\vert\mathcal{X}_{1}\land\mathcal{X}_{2}]\mathbbm{1}
_{Y}\vert(\mathcal{X}_{1}\lor\mathcal{Y})\land(\mathcal{X}_{2}
\lor\mathcal{Y})]=\mathbb{P}[X\vert\mathcal{X}_{1}\land\mathcal{X}
_{2}]\mathbbm{1}_{Y}\in((\mathcal{X}_{1}\land\mathcal{X}_{2})\lor
\mathcal{Y})/\mathcal{B}_{[-\infty,\infty]}$. Again a $\pi/\lambda
$-argument allows to conclude.

\ref{dist:b}. Relabel $\mathcal{X}_{\alpha\alpha}=:\mathcal{X}_{
\alpha}$ and $\mathcal{Z}_{\alpha}=:\mathcal{A}_{\alpha}$,
$\alpha\in\mathfrak{A}$. Suppose \eqref{eq:dist} has been shown for
$\mathfrak{A}$ finite (all the time assuming \ref{dist:b}). Let $T$
consist of the finite subsets of $\mathfrak{A}$, direct $T$ by inclusion
$\subset$, and define $\overline{\mathcal{X}}_{A}:=\lor_{\alpha
\in A}\mathcal{X}_{\alpha}$ for $A\in T$. Then $
\land_{\alpha\in\mathfrak{A}}(\mathcal{X}_{\alpha}\lor(
\lor_{\alpha'\in\mathfrak{A}\backslash\{\alpha\}}\mathcal{A}_{
\alpha'}))=\land_{A\in T}(\overline{\mathcal{X}}_{A}\lor(
\lor_{\alpha'\in\mathfrak{A}\backslash A}\mathcal{A}_{\alpha'}))$.
Now let $B\in T\backslash\{\emptyset\}$, $A_{i}\in\mathcal{A}_{i}$
for $i\in B$. We have by decreasing martingale convergence,\index{martingale convergence} a.s.
$\mathbb{P}[\cap_{i\in B}A_{i}\vert\land_{A\in T}(\overline{
\mathcal{X}}_{A}\lor(\lor_{\alpha'\in\mathfrak{A}\backslash A}
\mathcal{A}_{\alpha'}))]=\lim_{A\in T} \mathbb{P}[\cap_{i\in B}A_{i}
\vert\overline{\mathcal{X}}_{A}\lor(
\lor_{\alpha'\in\mathfrak{A}\backslash A}\mathcal{A}_{\alpha'})]=
\mathbb{P}[\cap_{i\in B}A_{i}\vert\overline{\mathcal{X}}_{B}]\in(
\lor_{\alpha\in\mathfrak{A}}\mathcal{X}_{\alpha})/\mathcal{B}_{[-
\infty,\infty]}$, where the limit is in $L^{1}(\mathbb{P})$, and we
conclude that \eqref{eq:dist} holds true via a $\pi/\lambda$-argument.
So it remains to argue \eqref{eq:dist} for $\mathfrak{A}$ finite, and
then by an inductive\vadjust{\goodbreak} argument for $\mathfrak{A}=\{1,2\}$, in which case
we are to establish that $(\mathcal{X}_{1}\lor\mathcal{A}_{2})\land
(\mathcal{A}_{1}\lor\mathcal{X}_{2})=\mathcal{X}_{1}\land
\mathcal{X}_{2}$. To this end let $F\in(\mathcal{X}_{1}\lor
\mathcal{A}_{2})\land(\mathcal{A}_{1}\lor\mathcal{X}_{2})$. Then a.s.
$\mathbbm{1}_{F}=\mathbb{P}[F\vert\mathcal{X}_{1}\lor\mathcal{A}
_{2}]$ (because $F\in\mathcal{X}_{1}\lor\mathcal{A}_{2}$), which is
$\in(\mathcal{X}_{1}\lor\mathcal{X}_{2})/\mathcal{B}_{[-\infty,
\infty]}$ (because $F\in\mathcal{A}_{1}\lor\mathcal{X}_{2}$, by a
$\pi/\lambda$-argument, using $\mathcal{X}_{1}\subset\mathcal{A}
_{1}$, $\mathcal{X}_{2}\subset\mathcal{A}_{2}$ and $\mathcal{A}_{2}
\perp\!\!\!\!\perp\mathcal{A}_{1}$: if $A_{1}\in\mathcal{A}_{1}$ and
$X_{2}\in\mathcal{X}_{2}$ then a.s. $\mathbb{P}[A_{1}\cap X_{2}
\vert\mathcal{X}_{1}\lor\mathcal{A}_{2}]=\mathbbm{1}_{X_{2}}
\mathbb{P}[A_{1}\vert\mathcal{X}_{1}]\in(\mathcal{X}_{1}\lor
\mathcal{X}_{2})/\mathcal{B}_{[-\infty,\infty]}$).
\end{proof}

%c3.8 #&#
%
\begin{corollary}[Distributivity II\index{distributivity}]
\label{lemma}
\leavevmode
\begin{enumerate}[(i)]%
\item
\label{lemma::ii}
If $\mathcal{Y}\in\varLambda$ is independent of a nonincreasing sequence
$(\mathcal{X}_{n})_{n\in\mathbb{N}}$ from $\varLambda$, then
$\land_{n\in\mathbb{N}}(\mathcal{X}_{n}\lor\mathcal{Y})=(
\land_{n\in\mathbb{N}}\mathcal{X}_{n})\lor\mathcal{Y}$.
\cite[Exercise~2.5(1-2)]{chaumont-yor},
\xch{\cite[Exercise~2.15]{revuz-yor}.}{\cite[Exercise~2.15]{revuz-yor}}
\item
\label{lemma::iv}
For $\{\mathcal{X}_{1},\mathcal{X}_{2},\mathcal{Y}_{1},\mathcal{Y}
_{2}\}\subset\varLambda$, if $\mathcal{X}_{1}\lor\mathcal{X}_{2}
\perp\!\!\!\!\perp\mathcal{Y}_{1}\lor\mathcal{Y}_{2}$, then
$(\mathcal{X}_{1}\lor\mathcal{Y}_{1})\land(\mathcal{X}_{2}\lor
\mathcal{Y}_{2})=(\mathcal{X}_{1}\land\mathcal{X}_{2})\lor(
\mathcal{Y}_{1}\land\mathcal{Y}_{2})$. \cite[Fact~2.18, when
$\mathcal{M}$ is countably generated up to negligible sets]{tsirelson}.
In particular for $\{\mathcal{X},\mathcal{A},\mathcal{Y}\}\subset
\varLambda$, if $\mathcal{X}\subset\mathcal{A}\perp\!\!\!\!\perp
\mathcal{Y}$, then $(\mathcal{X}\lor\mathcal{Y})\land\mathcal{A}=
\mathcal{X}$.
\item
\label{lemma:i}
If $\{\mathcal{X},\mathcal{Y},\mathcal{Z}\}\subset\varLambda$,
$\mathcal{X}\lor\mathcal{Y}\perp\!\!\!\!\perp\mathcal{Z}$, then
$(\mathcal{X}\lor\mathcal{Z})\land(\mathcal{Y}\lor\mathcal{Z})=(
\mathcal{X}\land\mathcal{Y})\lor\mathcal{Z}$. \qed
\end{enumerate}
\end{corollary}

%
%r3.9 #&#
%
\begin{remark}
\cite{weiz} discusses the equality in \ref{lemma:ii} when
$\mathcal{X}$ and $\mathcal{Y}$ are not necessarily independent; we have
seen in Example~\ref{example:dist}\ref{example:warning} that it fails
in general.
\end{remark}
%
%r3.10 #&#
%
\begin{remark}
In \ref{lemma:i} the equality $(\mathcal{X}\land\mathcal{Z})\lor(
\mathcal{Y}\land\mathcal{Z})= (\mathcal{X}\lor\mathcal{Y})\land
\mathcal{Z}$ is trivial (both sides are equal to $0_{\varLambda}$).
Example~\ref{example:dist}\ref{dist:basic} showed that these basic
distributivity relations fail in general, even when  $\mathcal{X},
\mathcal{Y},\mathcal{Z}$ are pairwise independent.\index{pairwise independent}
\end{remark}
%
%r3.11 #&#
%
\begin{remark}
Let $\{\mathcal{A},\mathcal{B},\mathcal{C}\}\subset\varLambda$. (I) If
$\mathcal{A}\subset\mathcal{B}\lor\mathcal{C}$ and $\mathcal{A}
\lor\mathcal{B}\perp\!\!\!\!\perp\mathcal{C}$, then $\mathcal{A}
\subset\mathcal{B}$: $\mathcal{A}=\mathcal{A}\land(\mathcal{B}
\lor\mathcal{C})= (\mathcal{A}\lor0_{\varLambda})\land(\mathcal{B}
\lor\mathcal{C})=\mathcal{A}\land\mathcal{B}$ by \ref{lemma::iv},
\cite[Exercise~2.2(1)]{chaumont-yor}. (II) If $\mathcal{A}\subset
\mathcal{B}\lor\mathcal{C}$, $\mathcal{A}\perp\!\!\!\!\perp
\mathcal{C}$, $\mathcal{B}\subset\mathcal{A}$, then $\mathcal{A}=
\mathcal{B}$: $\mathcal{A}\subset(\mathcal{B}\lor\mathcal
{C})\land
(\mathcal{A}\lor0_{\varLambda})=\mathcal{B}$ by \ref{lemma::iv} again,
\cite[Exercise~2.2(3)]{chaumont-yor}.
\end{remark}

We turn now to complements;\index{complements} we shall resume with the investigation of
distributivity\index{distributivity} later on in
Nos.~\ref{proposition:distIII}-\ref{corollary:VI}.

%p3.12 #&#
%
\begin{proposition}[Complements\index{complements} I]
\label{proposition:complements}%
\textup{\cite[Proposition~4]{emery-sch}}. Let $\{\mathcal{X},\mathcal{Y}\}
\subset\varLambda$. Assume $\mathcal{X}$ is countably generated up to
negligible sets and $\mathcal{Y}\subset\mathcal{X}$. Then the following
statements are equiveridical.
\begin{enumerate}[(i)]%
\item
\label{compl:i'}
Whenever $\mathsf{X}\in\mathcal{X}/\mathcal{B}_{\mathbb{R}}$ is such
that $\mathcal{X}=\overline{\sigma}(\mathsf{X})$, then for every
$\mathsf{Y}\in\mathcal{Y}/\mathcal{B}_{\mathbb{R}}$, $\mathbb{P}(
\mathsf{X}=\mathsf{Y})=0$.
\item
\label{compl:i}
There exists $\mathsf{X}\in\mathcal{X}/\mathcal{B}_{\mathbb{R}}$ such
that for every $\mathsf{Y}\in\mathcal{Y}/\mathcal{B}_{\mathbb{R}}$,
$\mathbb{P}(\mathsf{X}=\mathsf{Y})=0$.
\item
\label{compl:ii}
There exists $\mathsf{Z}\in\mathcal{X}/\mathcal{B}_{\mathbb{R}}$
independent of $\mathcal{Y}$ and having a diffuse law.\index{diffuse}
\item
\label{compl:iii}
There exists $\mathsf{Z}\in\mathcal{X}/\mathcal{B}_{[0,1]}$ independent
of $\mathcal{Y}$ with uniform law such that $\mathcal{Y}+\overline{
\sigma}(\mathsf{Z})=\mathcal{X}$.
\item
\label{compl:iv}
Every $\mathsf{Z}\in\mathcal{X}/\mathcal{B}_{\mathbb{R}}$ for which
$\mathcal{Y}\lor\overline{\sigma}(\mathsf{Z})=\mathcal{X}$ has a
diffuse law.\index{diffuse}
\end{enumerate}
\end{proposition}

%d3.13 #&#
%
\begin{definition}
Let $\{\mathcal{X},\mathcal{Y}\}\subset\varLambda$, $\mathcal
{Y}\subset
\mathcal{X}$, $\mathcal{X}$ countably generated up to negligible sets.
Following \cite{emery-sch} call $\mathcal{X}$ conditionally
non-atomic given $\mathcal{Y}$ when the conditions
\ref{compl:i'}-\ref{compl:iv} of
Proposition~\ref{proposition:complements} prevail.
\end{definition}

%e3.14 #&#
%
\begin{example}
\label{example:tweak}
Let $\{\mathcal{A},\mathcal{B},\mathcal{X}\}\subset\varLambda$,
$\mathcal{X}\subset\mathcal{A}+\mathcal{B}$. It can happen that
$\mathcal{A}$, $\mathcal{B}$, $\mathcal{X}$ are pairwise independent\index{pairwise independent}
\cite[Exercise~2.1(3)]{chaumont-yor}, and even when it is so, it may
then happen that\vadjust{\goodbreak} there is no $\mathcal{X}'\in\varLambda$ with
$\mathcal{X}'\subset\mathcal{B}$ and $\mathcal{A}+ \mathcal{X}=
\mathcal{A}+\mathcal{X}'$, i.e. $\mathcal{X}\subset((\mathcal{A}
\lor\mathcal{X})\land\mathcal{B})\lor\mathcal{A}$ may fail (in
particular one can have $\mathcal{X}$ independent of $\mathcal{B}$, but
not measurable w.r.t. $\mathcal{A}$
\cite[Exercise~2.1(2)]{chaumont-yor}). In the ``discrete''
setting\footnote{In precise terms, by ``discrete'', we mean here, and
in what follows, that every $\sigma$-field under consideration is
generated up to negligible sets by a discrete random variable.} take,
e.g., $\xi_{i}$, $i\in\{1,2,3,4\}$, independent equiprobable signs.\index{independent ! equiprobable signs} Let
$\mathcal{A}=\overline{\sigma}(\xi_{1},\xi_{2})$, $\mathcal
{B}=\overline{
\sigma}(\xi_{3},\xi_{4})$, $\mathcal{X}=\overline{\sigma}(\xi
_{1}\xi
_{3}+\xi_{2}\xi_{4})$. Then it is mechanical to check that $(
\mathcal{X}\lor\mathcal{A})\land\mathcal{B}=\overline{\sigma
}(\xi
_{3}\xi_{4})$ (e.g.,
for inclusion $\supset$ one can notice that $(\xi_{1}\xi
_{3}+\xi_{2}\xi_{4})^{2}=2(1+\xi_{1}\xi_{2}\xi_{3}\xi_{4})$; for the
reverse inclusion one can consider the behavior of the indicators of the
elements of $\sigma(\xi_{3},\xi_{4})$ on the atoms of $\sigma(\xi
_{1},\xi_{2},\xi_{1}\xi_{3}+\xi_{2}\xi_{4})$). But $\xi_{1}\xi
_{3}+\xi
_{2}\xi_{4}$ is not measurable w.r.t. $\mathcal{A}\lor((\mathcal{A}
\lor\mathcal{X})\land\mathcal{B}\xch{)}{))}=\overline{\sigma}(\xi_{1},\xi
_{2},\xi_{3}\xi_{4})$, indeed $\xi_{1}\xi_{3}+\xi_{2}\xi_{4}$ is
not \xch{a.s.}{a.s}
constant on the atom $\{\xi_{1}=1,\xi_{2}=1,\xi_{3}\xi_{4}=1\}$ of
$\sigma(\xi_{1},\xi_{2},\xi_{3}\xi_{4})$. To tweak this to the
``continuous'' case,\footnote{To be precise, by ``continuous'', we mean
to say here, and in what follows, that every $\sigma$-field under
consideration is generated up to negligible sets by a diffuse random
variable.} simply take a sequence $(\xi_{i})_{i\in\mathbb{N}}$ of
independent equiprobable signs\index{independent ! equiprobable signs} and set $\mathcal{A}=\overline{\sigma
}(\xi_{2i}:i\in\mathbb{N})$, $\mathcal{B}=\overline{\sigma}(\xi_{2i+1}:i
\in\mathbb{N}_{0})$, $\mathcal{X}=\overline{\sigma}(\xi_{1}\xi_{2}+
\xi_{3}\xi_{4},\xi_{5}\xi_{6}+\xi_{7}\xi_{8},\ldots)$. By
Proposition~\ref{distributivity} and the preceding, it follows that
$(\mathcal{X}\lor\mathcal{A})\land\mathcal{B}=\overline{\sigma
}(\xi
_{1}\xi_{3},\xi_{5}\xi_{7},\ldots)$, and we see that $\xi_{1}\xi_{2}+
\xi_{3}\xi_{4}$ is not measurable w.r.t. $((\mathcal{X}\lor
\mathcal{A})\land\mathcal{B})\lor\mathcal{A}$, for, exactly as before,
it is not measurable w.r.t. $\overline{\sigma}(\xi_{2},\xi_{4},\xi
_{1}\xi_{3})=[((\mathcal{X}\lor\mathcal{A})\land\mathcal{B})\lor
\mathcal{A}]\land\overline{\sigma}(\xi_{1},\ldots,\xi_{4})$.\looseness=1
\end{example}

%e3.15 #&#
%
\begin{examples}
Let $\{\mathcal{X},\mathcal{Y}\}\subset\varLambda$, $\mathcal
{Y}\subset
\mathcal{X}$.\vspace*{-1.5pt}
\begin{enumerate}[(a)]%
\item
We have already seen in Example~\ref{example:not-exist} that in general
$\mathcal{Y}$ may fail to have a complement in $\mathcal{X}$, though by
Proposition~\ref{proposition:complements} this cannot happen when
$\mathcal{X}$ is essentially separable and everything is ``sufficiently
continuous''. Example~\ref{example:not-unique} shows, in a ``discrete''
setting, that even when $\mathcal{Y}$ has a complement in $
\mathcal{X}$, then it is not necessarily unique. To see the latter also
in the ``continuous'' setting take a doubly infinite sequence
$(\xi_{i})_{i\in\mathbb{Z}}$ of independent equiprobable signs,\index{independent ! equiprobable signs} and set
$\mathcal{X}=\overline{\sigma}(\xi_{i}:i\in\mathbb{Z})$, $
\mathcal{Y}=\overline{\sigma}(\xi_{i}:i\in\mathbb{N})$. Then
$\mathcal{Y}+\overline{\sigma}(\xi_{i}:i\in\mathbb{Z}_{\leq0})=
\mathcal{X}$ but also $\mathcal{Y}+\overline{\sigma}(\xi_{i}\xi_{i+1}:i
\in\mathbb{Z}_{\leq0})=\mathcal{X}$.\vspace*{-1.5pt}
\item
Even when the equivalent conditions of
Proposition~\ref{proposition:complements} are met, and a $\mathcal{Z}
\in\varLambda$ satisfies $\mathcal{Y}\lor\mathcal{Z}=\mathcal{X}$, there
may be no $\mathcal{Z}'\in\varLambda$ with $\mathcal{Z}'\subset
\mathcal{Z}$ and $\mathcal{Y}+\mathcal{Z}'=\mathcal{X}$. The following
example of this situation is essentially verbatim from \cite[p. 11,
Remark (b)]{emery-sch}. Let $\varOmega=([0,\frac{1}{2}]\times[0,1])
\cup([\frac{1}{2},1]\times[0,\frac{1}{2}])\cup([1,\frac{3}{2}]
\times[\frac{1}{2},1])$, $\mathcal{M}=\mathcal{B}_{\varOmega}$, and
$\mathbb{P}$ be the (restriction of the) Lebesgue measure. Let
$\mathsf{Y}$ be the projection onto the first coordinate and
$\mathsf{Z}$ be the projection onto the second coordinate, $
\mathcal{Y}=\overline{\sigma}(\mathsf{Y})$, $\mathcal{Z}=\overline{
\sigma}(\mathsf{Z})$, $\mathcal{X}=\overline{\sigma}(\mathsf{Y},
\mathsf{Z})=\mathcal{M}$. Then $\vert\mathsf{Z}-\frac{1}{2}\vert$ is
independent of $\mathcal{Y}$, verifying \ref{compl:ii}, though
$\mathcal{Y}$ and $\mathcal{Z}$ are not independent. Suppose that
$\mathcal{Z}'\in\varLambda$ satisfies $\mathcal{Z}'\subset\mathcal{Z}$
and $\mathcal{Y}\lor\mathcal{Z}'=\mathcal{X}$. The $\sigma$-field
$\mathcal{X}$ and hence $\mathcal{Z}'$ is countably generated up to
negligible sets so there is $\mathsf{Z}'\in\mathcal{Z}'/
\mathcal{B}_{\mathbb{R}}$ such that $\mathcal{Z}'=\overline{\sigma}(
\mathsf{Z}')$. By the Doob--Dynkin lemma there are $f\in\mathcal{B}
_{[0,1]}/\mathcal{B}_{\mathbb{R}}$ and $g\in\mathcal{B}_{[0,
\frac{3}{2}]\times\mathbb{R}}/\mathcal{B}_{[0,1]}$ such that a.s.
$\mathsf{Z}'=f(\mathsf{Z})$ and $\mathsf{Z}=g(\mathsf{Y},\mathsf{Z}')$.
Then $\mathsf{Z}=g(\mathsf{Y},f(\mathsf{Z}))$ a.s.; consequently by
Tonelli's theorem for Lebesgue-almost every $y\in[0,\frac{1}{2}]$,
$z=g(y,f(z))$ for Lebesgue-almost all $z\in[0,1]$. Fix such $y$. Then
because $\mathsf{Z}$ is absolutely continuous, one obtains $
\mathsf{Z}=g(y,f(\mathsf{Z}))=g(y,\mathsf{Z}')$ a.s.; this forces
$\mathcal{Z}'=\mathcal{Z}$, preventing $\mathcal{Z}'\perp\!\!\!\!
\perp\mathcal{Y}$.
\item
If the equivalent conditions of
Proposition~\ref{proposition:complements} are met and if $\mathsf{Z}
\in\mathcal{X}/\mathcal{B}_{\mathbb{R}}$ has diffuse law\index{diffuse} and is
independent of $\mathcal{Y}$, there may exist no $\mathcal{Z}'\in
\varLambda$ such that $\mathcal{Y}+\mathcal{Z}'=\mathcal{X}$ and
$\overline{\sigma}(\mathsf{Z})\subset\mathcal{Z}'$ (however this
cannot happen if ceteris paribus $\mathsf{Z}$ is discrete rather than
continuous -- see
Corollary~\ref{proposition:non-atomic}\ref{non-atomic:ii}\ref
{non-atomic:ii:b}).
We repeat here for the reader's convenience
\cite[p. 11, Remark (a)]{emery-sch} exemplifying this scenario. Let
$\mathsf{X},\mathsf{Y},\mathsf{Z}$ be independent random variables\index{independent ! random variables} with
uniform law on $[0,1]$ and let $\mathcal{Y}=\overline{\sigma}(
\mathsf{Y})$, $\mathcal{X}=\overline{\sigma}(\mathsf{Y},\mathsf{Z},
\mathsf{X}\mathbbm{1}_{\{\mathsf{Y}< \frac{1}{2}\}})$. Clearly
$\mathcal{X}$ is countably generated up to negligible sets;
$\mathsf{Z}$ has a diffuse law\index{diffuse} and is independent of $\mathcal{Y}$; in
particular \ref{compl:ii} is verified. Let $\mathcal{Z}'\in\varLambda$
be such that $\mathcal{Y}\perp\!\!\!\!\perp\mathcal{Z}'\supset
\overline{
\sigma}(\mathsf{Z})$, $\mathcal{Z}'\subset\mathcal{X}$. There is a
$\mathsf{Z}'\in\mathcal{Z}'/\mathcal{B}_{\mathbb{R}}$ such that
$\mathcal{Z}'=\overline{\sigma}(\mathsf{Z}')$. By the Doob--Dynkin lemma
there are $f\in\mathcal{B}_{\mathbb{R}}/\mathcal{B}_{[0,1]}$ and
$g\in\mathcal{B}_{[0,1]^{3}}/\mathcal{B}_{\mathbb{R}}$ such that a.s.
$\mathsf{Z}=f(\mathsf{Z}')$ and $\mathsf{Z}'=g(\mathsf{Y},\mathsf{Z},
\mathsf{X}\mathbbm{1}_{\{\mathsf{Y}< \frac{1}{2}\}})$. Then on
$\{\mathsf{Y}\geq\frac{1}{2}\}$, $\mathsf{Z}'=g(\mathsf{Y},
\mathsf{Z},0)=g(\mathsf{Y},f(\mathsf{Z}'),0)$ a.s.; hence by Tonelli's
theorem for Lebesgue-almost every $y\in[\frac{1}{2},1]$,
$z'=g(y,f(z'),0)$ for $\mathsf{Z}'_{\star}\mathbb{P}$-almost every
$z'\in\mathbb{R}$. Fix such $y$. It follows that $\mathsf{Z}'=g(y,f(
\mathsf{Z}'),0)=g(y,\mathsf{Z},0)$ a.s.; this forces $\mathcal
{Z}'=\overline{
\sigma}(\mathsf{Z})$, which precludes $\mathcal{Y}\lor\mathcal{Z}'=
\mathcal{X}$.
\end{enumerate}
\end{examples}

\begin{proof}[Proof of Proposition~\ref{proposition:complements}]
We follow closely the proof of \cite[Proposition~4]{emery-sch}.

\ref{compl:i'} $\Rightarrow$ \ref{compl:i} because $\mathcal{X}$ is
countably generated up to negligible sets.

\ref{compl:iii} $ \Rightarrow$ \ref{compl:ii} is trivial.

\ref{compl:ii} $\Rightarrow$ \ref{compl:i} by Tonelli's theorem.

\ref{compl:iv} $\Rightarrow$ \ref{compl:i'}. Let $\mathsf{X}\in
\mathcal{X}/\mathcal{B}_{\mathbb{R}}$ be such that $\mathcal
{X}=\overline{
\sigma}(\mathsf{X})$, take $\mathsf{Y}\in\mathcal{Y}/\mathcal{B}
_{\mathbb{R}}$. Fix $x_{0}\in\mathbb{R}$ for which $\mathbb{P}(
\mathsf{X}=x_{0})=0$. Then $\mathcal{Y}\lor\overline{\sigma}(
\mathsf{X}\mathbbm{1}_{\{\mathsf{X}\ne\mathsf{Y}\}}+x_{0}\mathbbm{1}
_{\{\mathsf{X}=\mathsf{Y}\}})=\mathcal{X}$, hence by \ref{compl:iv}
$\mathsf{X}\mathbbm{1}_{\{\mathsf{X}\ne\mathsf{Y}\}}+x_{0}
\mathbbm{1}_{\{\mathsf{X}=\mathsf{Y}\}}$ has a diffuse law,\index{diffuse} and
therefore $\mathbb{P}(\mathsf{X}=\mathsf{Y})=0$.

\ref{compl:i} $\Rightarrow$ \ref{compl:iv}. Let $\mathsf{X}\in
\mathcal{X}/\mathcal{B}_{\mathbb{R}}$ be such that for every
$\mathsf{Y}\in\mathcal{Y}/\mathcal{B}_{\mathbb{R}}$, $\mathbb{P}(
\mathsf{X}=\mathsf{Y})=0$ and let $\mathsf{Z}\in\mathcal{X}/
\mathcal{B}_{\mathbb{R}}$ be such that $\mathcal{Y}\lor\overline{
\sigma}(\mathsf{Z})=\mathcal{X}$. Because $\mathcal{Y}$ is countably
generated up to negligible sets, there is $\mathsf{Y}\in
\mathcal{Y}/\mathcal{B}_{\mathbb{R}}$ such that $\mathcal
{Y}=\overline{
\sigma}(\mathsf{Y})$. Then $\overline{\sigma}(\mathsf{Y},\mathsf{Z})=
\mathcal{X}$ and by the Doob--Dynkin lemma there is $f\in
\mathcal{B}_{\mathbb{R}^{2}}/\mathcal{B}_{\mathbb{R}}$ such that a.s.
$\mathsf{X}=f(\mathsf{Y},\mathsf{Z})$. We conclude that for each
$z_{0}\in\mathbb{R}$, $\mathbb{P}(\mathsf{Z}=z_{0})\subset
\mathbb{P}(\mathsf{X}=f(\mathsf{Y},z_{0}))=0$.

\ref{compl:i} $\Rightarrow$ \ref{compl:iii}. Let again $\mathsf{X}
\in\mathcal{X}/\mathcal{B}_{\mathbb{R}}$ be such that for every
$\mathsf{Y}\in\mathcal{Y}/\mathcal{B}_{\mathbb{R}}$, $\mathbb{P}(
\mathsf{X}=\mathsf{Y})=0$. Take also $\mathsf{Y}\in\mathcal{Y}/
\mathcal{B}_{\mathbb{R}}$ such that $\mathcal{Y}=\overline{\sigma}(
\mathsf{Y})$ and $\mathsf{X}'\in\mathcal{X}/\mathcal{B}_{
\mathbb{R}}$ such that $\overline{\sigma}(\mathsf{X}')=\mathcal{X}$.
Let $\mu$ be the law of $\mathsf{Y}$ and let $(\nu_{y})_{y\in
\mathbb{R}}$ be a version of the conditional law of $\mathsf{X}'$ given
$\mathsf{Y}$: $(\mathbb{R}\ni y\mapsto\nu_{y}(A))\in\mathcal{B}_{
\mathbb{R}}/\mathcal{B}_{[0,1]}$ for each $A\in\mathcal{B}_{
\mathbb{R}}$; $\nu_{y}$ is a law on $\mathcal{B}_{\mathbb{R}}$ for each
$y\in\mathbb{R}$; and $\mathbb{E}[f(\mathsf{X}',\mathsf{Y})]=\int f(x',y)
\nu_{y}(dx')\mu(dy)$ for $f\in\mathcal{B}_{\mathbb{R}^{2}}/
\mathcal{B}_{[0,\infty]}$. Remark that in particular $(\star)$ a.s.
$\mathsf{X}'$ cannot fall into a maximal non-degenerate interval that
is negligible for $\nu_{Y}$. Besides, by the Doob--Dynkin lemma, there
is $g\in\mathcal{B}_{\mathbb{R}}/\mathcal{B}_{\mathbb{R}}$ such that
$\mathsf{X}=g(\mathsf{X}')$ a.s. Then $\mathbb{P}(\mathsf{Y}'=
\mathsf{X}')\subset\mathbb{P}(\mathsf{X}=g(\mathsf{Y}'))=0$ for any
$\mathsf{Y}'\in\mathcal{Y}/\mathcal{B}_{\mathbb{R}}$. From this it
follows that $(\star\star)$ $\nu_{y}$ is diffuse\index{diffuse} for $\mu$-almost
every $y\in\mathbb{R}$. Set now $\mathsf{Z}:=\nu_{\mathsf{Y}}((-
\infty,\mathsf{X}'])\in\mathcal{X}/\mathcal{B}_{[0,1]}$; then for
$\phi\in\mathcal{B}_{\mathbb{R}}/\mathcal{B}_{[0,\infty]}$ and
$z\in[0,1]$,
\begin{align*}
\mathbb{E}\bigl[\phi(\mathsf{Y});\mathsf{Z}\leq z\bigr]
&{}=\int\int\phi(y)
\mathbbm{1}_{[0,z]}(\nu_{y}\bigl(\bigl(-\infty,x']\bigr)\bigr)\nu_{y}\bigl(dx'\bigr)\mu(dy)\\
&{}=z\int\phi d\mu=\mathbb{P}(\mathsf{Z}\leq z)\mathbb{E}\bigl[\phi( \mathsf{Y})\bigr],
\end{align*}
because of $(\star\star)$. On account of $(\star)$, it also follows
from the equality $\mathsf{Z}=\nu_{Y}((-\infty,\break \mathsf{X}'])$ that
$\mathsf{X}'\in\overline{\sigma}(\mathsf{Z},\mathsf{Y})$. Thus
$\mathsf{Z}$ meets all the requisite properties.
\end{proof}
Several ``stability'' properties of conditionally non-atomic
$\sigma$-fields can be noted:

%c3.16 #&#
%
\begin{corollary}[Conditionally non-atomic $\sigma$-fields]
\label{proposition:non-atomic}%
\textup{\cite[Corollaries~3 and~4]{emery-sch}} Let $\{\mathcal{X},
\mathcal{Y},\mathcal{Z}\}\subset\varLambda$, $\mathcal{Y}\subset
\mathcal{X}$. Assume $\mathcal{X}\lor\mathcal{Z}$ is countably
generated up to negligible sets.
\begin{enumerate}[(i)]%
\item
\label{non-atomic:i}
If $\mathcal{X}\lor\mathcal{Z}$ is conditionally non-atomic given
$\mathcal{Y}\lor\mathcal{Z}$, then $\mathcal{X}$ is conditionally
non-atomic given $\mathcal{Y}$.
\item
\label{non-atomic:ii}
Suppose $\mathcal{X}$ is conditionally non-atomic given $\mathcal{Y}$.
\begin{enumerate}[(a)]%
\item
\label{non-atomic:ii:a}
If $\mathcal{X}$ and $\mathcal{Z}$ are independent, then $\mathcal{X}
\lor\mathcal{Z}$ is conditionally non-atomic given $\mathcal{Y}
\lor\mathcal{Z}$.
\item
\label{non-atomic:ii:b}
If $\mathcal{P}\subset\mathcal{X}$ is a denumerable partition\index{denumerable ! partition} of
$\varOmega$, then $\mathcal{X}$ is conditionally non-atomic given
$\mathcal{Y}\lor\overline{\sigma}(\mathcal{P})$; if further
$\overline{\sigma}(\mathcal{P})\perp\!\!\!\!\perp\mathcal{Y}$, then
there exists $\mathsf{Z}\in\mathcal{X}/\mathcal{B}_{[0,1]}$ with
uniform law such that $\mathcal{Y}+ \overline{\sigma}(\mathsf{Z})=
\mathcal{X}$ and $\overline{\sigma}(\mathcal{P})\subset\overline{
\sigma}(\mathsf{Z})$.
\end{enumerate}
\end{enumerate}
\end{corollary}
\begin{proof}
We follow closely the proofs of
\cite[Corollaries~3 and~4]{emery-sch}.

\ref{non-atomic:i}. Let $\mathsf{Z}\in\mathcal{X}/\mathcal{B}_{
\mathbb{R}}$ be such that $\mathcal{X}=\mathcal{Y}\lor\overline{
\sigma}(\mathsf{Z})$; then $\mathcal{X}\lor\mathcal{Z}=(\mathcal{Y}
\lor\mathcal{Z})\lor\overline{\sigma}(\mathsf{Z})$. Thus if
$\mathcal{X}\lor\mathcal{Z}$ is conditionally non-atomic given
$\mathcal{Y}\lor\mathcal{Z}$, then by
Proposition~\ref{proposition:complements}\ref{compl:iv} $\mathsf{Z}$ is
diffuse,\index{diffuse} which makes $\mathcal{X}$ conditionally non-atomic given
$\mathcal{Y}$ by the very same argument.
%token.

\ref{non-atomic:ii}\ref{non-atomic:ii:a}. Let $\mathcal{X}$ and
$\mathcal{Z}$ be independent. By
Proposition~\ref{proposition:complements}\ref{compl:ii}, there exists
$\mathsf{Z}\in\mathcal{X}/\mathcal{B}_{\mathbb{R}}$ independent of
$\mathcal{Y}$ and having a diffuse law;\index{diffuse} such $\mathsf{Z}$ is then also
independent of $\mathcal{Y}\lor\mathcal{Z}$, so that by the very same
condition $\mathcal{X}\lor\mathcal{Z}$ is conditionally non-atomic
given $\mathcal{Y}\lor\mathcal{Z}$.

\ref{non-atomic:ii}\ref{non-atomic:ii:b}. There is a random variable
$\mathsf{P}\in\mathcal{X}/2^{\mathbb{N}}$ for which $\overline{
\sigma}(\mathcal{P})=\overline{\sigma}(\mathsf{P})$. If $\mathsf{Z}
\in\mathcal{X}/\mathcal{B}_{\mathbb{R}}$ is such that $\mathcal{X}=(
\mathcal{Y}\lor\overline{\sigma}(\mathsf{P}))\lor\overline
{\sigma
}(\mathsf{Z})=\mathcal{Y}\lor\overline{\sigma}(\mathsf{P},
\mathsf{Z})$, then $(\mathsf{P},\mathsf{Z})$ has a diffuse law\index{diffuse} by
Proposition~\ref{proposition:complements}\ref{compl:iv}, hence (because
$\mathsf{P}$ has a denumerable range\index{denumerable ! range}) $\mathsf{Z}$ has a diffuse law,\index{diffuse}
which entails the desired conclusion by the very same
argument. %token.
Now suppose
$\mathsf{P}$ is independent of $\mathcal{Y}$. Via
Proposition~\ref{proposition:complements}\ref{compl:iii} let
$\mathsf{Z}'\in\mathcal{X}/\mathcal{B}_{[0,1]}$ have uniform law and
be a complement for $\mathcal{Y}+\overline{\sigma}(\mathsf{P})$ in
$\mathcal{X}$. Of course $\overline{\sigma}(\mathsf{Z}',\mathsf{P})$
is essentially separable so there is $\mathsf{Z}\in\sigma(
\mathsf{Z}',\mathsf{P})/\mathcal{B}_{\mathbb{R}}$ with $\overline{
\sigma}(\mathsf{Z})=\overline{\sigma}(\mathsf{Z}',\mathsf{P})$.
$\mathsf{Z}$ is diffuse,\index{diffuse} because $\mathsf{Z}'$ is, hence may be chosen
to be uniform on $[0,1]$.
\end{proof}
The next proposition investigates to what extent complements\index{complements} are
``hereditary''.
%
%p3.17 #&#
%
\begin{proposition}[Complements II\index{complements}]
\label{proposition:complementsII}
Let $\{\mathcal{X},\mathcal{Y},\mathcal{Z}\}\subset\varLambda$,
$\mathcal{Z}\subset\mathcal{X}+\mathcal{Y}$. Then the following
statements are equivalent.
\begin{enumerate}[(i)]%
\item
\label{complementII:i}
$\mathcal{Z}=(\mathcal{X}\land\mathcal{Z}) \lor(\mathcal{Y}\land
\mathcal{Z})$, i.e. $\mathcal{X}\land\mathcal{Z}$ is a complement of
$\mathcal{Y}\land\mathcal{Z}$ in $\mathcal{Z}$.
\item
\label{complementII:ii}
$\mathcal{X}$ and $\mathcal{Y}$ are conditionally independent\index{conditional independence} given
$\mathcal{Z}$, and $\mathbb{P}[Y\vert\mathcal{Z}]\in\mathcal{Y}/
\mathcal{B}_{[-\infty,\infty]}$ for $Y\in\mathcal{Y}$, $\mathbb{P}[X
\vert\mathcal{Z}]\in\mathcal{X}/\mathcal{B}_{[-\infty,\infty]}$ for
$X\in\mathcal{X}$.
\end{enumerate}
\end{proposition}
%
%r3.18 #&#
%
\begin{remark}
Dropping, ceteris paribus, the condition that $\mathcal{X}\perp\!\!
\!\!\perp\mathcal{Y}$, then \ref{complementII:i} no longer implies
\ref{complementII:ii} (because one can have $\mathcal{Z}\subset
\mathcal{X}$ or $\mathcal{Z}\subset\mathcal{Y}$, without $
\mathcal{X}$ and $\mathcal{Y}$ being conditionally independent\index{conditional independence} given
$\mathcal{Z}$); however, \ref{complementII:ii} still implies
\ref{complementII:i} (this will be clear from the proof, and at any rate
Proposition~\ref{proposition:distIV} will provide a more general
statement, that will subsume this implication as a special case).\vadjust{\goodbreak}
\end{remark}
%
%e3.19 #&#
%
\begin{examples}
\leavevmode
\begin{enumerate}[(a)]%
\item
\label{complII:a}
The situation described by \ref{complementII:i}, equivalently
\ref{complementII:ii} is not trivial. For instance if $\mathcal{A},
\mathcal{B},\mathcal{C},\mathcal{D}$ are independent members of
$\varLambda$, then one can take $\mathcal{X}=\mathcal{A}+\mathcal{B}$,
$\mathcal{Y}=\mathcal{C}+\mathcal{D}$, $\mathcal{Z}=\mathcal{B}+
\mathcal{C}$. Of course in this case $\mathcal{Z}=(\mathcal{X}\land
\mathcal{Z}) \lor(\mathcal{Y}\land\mathcal{Z})$ can be seen (slightly
indirectly) from Proposition~\ref{distributivity} as much as (directly)
from the validity of \ref{complementII:ii}.
\item
\label{complII:b}
But there are cases when Proposition~\ref{distributivity} does not apply
(or applies only (very) indirectly), while
Proposition~\ref{proposition:complementsII} does. A trivial example of
this is when $\mathcal{Z}\subset\mathcal{X}$ or $\mathcal{Z}\subset
\mathcal{Y}$.
\item
For a less trivial example of the situation described in
\ref{complII:b} let $\xi_{i}$, $i\in\{1,2,3,4\}$, be independent
equiprobable signs.\index{independent ! equiprobable signs} Let $\mathcal{X}=\overline{\sigma}(\xi_{1},\{
\xi
_{1}=\xi_{2}=1\})$, $\mathcal{Y}=\overline{\sigma}(\xi_{3},\{\xi_{3}=
\xi_{4}=1\})$ and $\mathcal{Z}=\overline{\sigma}(\xi_{1},\xi
_{3})$. In
this case, unlike in \ref{complII:a}, it is not the case that
$\mathcal{Z}\land\mathcal{X}=\overline{\sigma}(\xi_{1})$ would have
a complement in $\mathcal{X}$ and $\mathcal{Z}\land\mathcal
{Y}=\overline{
\sigma}(\xi_{3})$ would have a complement in $\mathcal{Y}$. For this
reason Proposition~\ref{distributivity} cannot be (indirectly) applied
to deduce $(\mathcal{X}\land\mathcal{Z})\lor(\mathcal{Y}\land
\mathcal{Z})=\mathcal{Z}$. Yet this equality does prevail and can indeed
be seen directly and a priori from the validity of~\ref{complementII:ii}.
\end{enumerate}
\end{examples}
\begin{proof}
Suppose \ref{complementII:i} hods true. Let $X\in\mathcal{X}$ and
$Y\in\mathcal{Y}$. Then because $\mathcal{X}\perp\!\!\!\!\perp
\mathcal{Y}$, a.s. $\mathbb{P}[X\cap Y\vert\mathcal{Z}]=\mathbb{P}[X
\cap Y\vert(\mathcal{X}\land\mathcal{Z})\lor(\mathcal{Y}\land
\mathcal{Z})]=\mathbb{P}[X\vert\mathcal{X}\land\mathcal{Z}]
\mathbb{P}[Y\vert\mathcal{Y}\land\mathcal{Z}]$. Taking $Y=\varOmega$ and
$X=\varOmega$ shows that $\mathbb{P}[X\vert\mathcal{Z}]=\mathbb{P}[X
\vert\mathcal{X}\land\mathcal{Z}]$ a.s. and $\mathbb{P}[Y\vert
\mathcal{Y}\land\mathcal{Z}]=\mathbb{P}[Y\vert\mathcal{Z}]$ a.s.,
which concludes the argument. Conversely, suppose that
\ref{complementII:ii} holds true. Let $X\in\mathcal{X}$ and
$Y\in\mathcal{Y}$. Then a.s. $\mathbb{P}[X\cap Y\vert\mathcal{Z}]=
\mathbb{P}[X\vert\mathcal{Z}]\mathbb{P}[Y\vert\mathcal{Z}]$ and
$\mathbb{P}[X\vert\mathcal{Z}]=\mathbb{P}[X\vert\mathcal{X}\land
\mathcal{Z}]$, $\mathbb{P}[Y\vert\mathcal{Z}]=\mathbb{P}[Y\vert
\mathcal{Y}\land\mathcal{Z}]$. Hence $\mathbb{P}[X\cap Y\vert
\mathcal{Z}]\in((\mathcal{X}\land\mathcal{Z})\lor(\mathcal{Y}
\land\mathcal{Z}))/\mathcal{B}_{[-\infty,\infty]}$. A~$\pi
/\lambda
$-argument allows to conclude that $\mathbb{P}[Z\vert\mathcal{Z}]
\in((\mathcal{X}\land\mathcal{Z})\lor(\mathcal{Y}\land\mathcal{Z}))/
\mathcal{B}_{[-\infty,\infty]}$ for all $Z\in\mathcal{X}\lor
\mathcal{Y}$ and therefore, because $\mathcal{Z}\subset\mathcal{X}
\lor\mathcal{Y}$, for all $Z\in\mathcal{Z}$. Thus $\mathcal{Z}
\subset(\mathcal{X}\land\mathcal{Z})\lor(\mathcal{Y}\land
\mathcal{Z})$, while the reverse inclusion is trivial.
\end{proof}
More generally (in the sufficiency part):
%
%p3.20 #&#
%
\begin{proposition}[Distributivity III\index{distributivity}]
\label{proposition:distIII}
Let $(\mathcal{X}_{\alpha})_{\alpha\in\mathfrak{A}}$ be a family in
$\varLambda$ consisting of independent $\sigma$-fields. Then
\begin{equation*}
(\lor_{\alpha\in\mathfrak{A}}\mathcal{X}_{\alpha})\land \mathcal{Z}=
\lor_{\alpha\in\mathfrak{A}}(\mathcal{X}_{\alpha} \land\mathcal{Z})
\end{equation*}
provided (i) the $\mathcal{X}_{\alpha}$, $\alpha\in\mathfrak{A}$, are
conditionally independent\index{conditional independence} given $\mathcal{Z}$ and (ii) $\mathbb{P}[X
_{\alpha}\vert\mathcal{Z}]\in\mathcal{X}_{\alpha}/\mathcal{B}_{[-
\infty,\infty]}$ for all $X_{\alpha}\in\mathcal{X}_{\alpha}$,
$\alpha\in\mathfrak{A}$.
\end{proposition}
\begin{proof}
Set $\mathcal{X}:=\lor_{\alpha\in\mathfrak{A}}\mathcal{X}_{\alpha}$.
Condition (ii) entails that a.s. $\mathbb{P}[X_{\alpha}\vert
\mathcal{X}\land\mathcal{Z}]=\mathbb{P}[X_{\alpha}\vert\mathcal{X}
_{\alpha}\land\mathcal{Z}]=\mathbb{P}[X_{\alpha}\vert\mathcal{Z}]$
for all $\alpha\in\mathfrak{A}$; combining this with (i) shows via a
$\pi/\lambda$-argument that a.s. $\mathbb{P}[X\vert\mathcal{X}
\land\mathcal{Z}]=\mathbb{P}[X\vert\mathcal{Z}]$ for all $X\in
\mathcal{X}$: if $B$ is a finite non-empty subset of $\mathfrak{A}$,
then a.s. $\mathbb{P}[\cap_{\beta\in B}X_{\beta}\vert\mathcal{Z}]=
\prod_{\beta\in B}\mathbb{P}[X_{\beta}\vert\mathcal{Z}]=
\prod_{\beta\in B}\mathbb{P}[X_{\beta}\vert\mathcal{X}\land
\mathcal{Z}]\in(\mathcal{X}\land\mathcal{Z})/\mathcal
{B}_{[-\infty
,\infty]}$. Replacing $\mathcal{Z}$ by $\mathcal{Z}\land\mathcal{X}$
if necessary, we may and do assume $\mathcal{Z}\subset\mathcal{X}$.
Then $\lor_{\alpha\in\mathfrak{A}}(\mathcal{X}_{\alpha}\land
\mathcal{Z})\subset\mathcal{Z}=\mathcal{X}\land\mathcal{Z}$ is
trivial. For the reverse inclusion, let $B$ be a finite non-empty subset
of $\mathfrak{A}$, and let $X_{\beta}\in\mathcal{X}_{\beta}$ for
$\beta\in B$. Then a.s. $\mathbb{P}[\cap_{\beta\in B}X_{\beta}
\vert\mathcal{Z}]=\prod_{\beta\in B}\mathbb{P}[X_{\beta}\vert
\mathcal{Z}]=\prod_{\beta\in B}\mathbb{P}[X_{\beta}\vert
\mathcal{X}_{\beta}\land\mathcal{Z}]\in(\lor_{\alpha\in
\mathfrak{A}}(\mathcal{X}_{\alpha}\land\mathcal{Z}))/\mathcal{B}
_{[-\infty,\infty]}$. By  a \mbox{$\pi/\lambda$-argument} we conclude that
$\mathbb{P}[Z\vert\mathcal{Z}]\in(\lor_{\alpha\in\mathfrak{A}}(
\mathcal{X}_{\alpha}\land\mathcal{Z}))/\mathcal{B}_{[-\infty,
\infty]}$ for all $Z\in\lor_{\alpha\in\mathfrak{A}}\mathcal{X}
_{\alpha}$, and therefore for all $Z\in\mathcal{Z}$. It means that
also $\mathcal{X}\land\mathcal{Z}=\mathcal{Z}\subset
\lor_{\alpha\in\mathfrak{A}}(\mathcal{X}_{\alpha}\land\mathcal{Z})$.
\end{proof}
Parallel to Proposition~\ref{proposition:distIII} we have:

%p3.21 #&#
%
\begin{proposition}[Distributivity IV\index{distributivity}]
\label{proposition:distIV}
Let $(\mathcal{X}_{\alpha})_{\alpha\in\mathfrak{A}}$ be a family in
$\varLambda$, with $\mathfrak{A}$ containing at least two elements,
consisting of $\sigma$-fields that are conditionally independent\index{conditional independence}
given $\mathcal{Z}\in\varLambda$. Then
\begin{equation*}
\mathcal{Z}=\land_{\alpha\in\mathfrak{A}}(\mathcal{X}_{\alpha} \lor
\mathcal{Z});
\end{equation*}
in particular $\land_{\alpha\in\mathfrak{A}}\mathcal{X}_{\alpha}
\subset\mathcal{Z}$.
\end{proposition}
%
%r3.22 #&#
%
\begin{remark}
The converse is not true, because, for instance, one can have
$\mathcal{X}$ and $\mathcal{Y}$ dependent with $\mathcal{X}\land
\mathcal{Y}=0_{\varLambda}$ (then $\mathcal{Z}=(\mathcal{X}\lor
\mathcal{Z})\land(\mathcal{Y}\lor\mathcal{Z})$ for $\mathcal{Z}=0_{
\varLambda}$, but $\mathcal{X}$ and $\mathcal{Y}$ are not independent
given $\mathcal{Z}$) -- see Example~\ref{ex:commute}. The condition on
the conditional independence\index{conditional independence} of course cannot be dropped, not even if
the $\mathcal{X}_{\alpha}$, $\alpha\in\mathfrak{A}$, and
$\mathcal{Z}$ are pairwise independent\index{pairwise independent} -- see
Example~\ref{example:dist}\ref{dist:basic}.
\end{remark}
%
%r3.23 #&#
%
\begin{remark}
By Proposition~\ref{distributivity} the equality
\begin{equation*}
(\land_{\alpha\in\mathfrak{A}}\mathcal{X}_{\alpha})\lor \mathcal{Z}=
\land_{\alpha\in\mathfrak{A}}(\mathcal{X}_{\alpha} \lor\mathcal{Z})
\end{equation*}
also prevails when the $\mathcal{X}_{\alpha}$, $\alpha\in
\mathfrak{A}$, are independent of $\mathcal{Z}$, however the scope of
this result is clearly different from that of
Proposition~\ref{proposition:distIV}.
\end{remark}

\begin{proof}
It is clear that $\mathcal{Z}\subset\land_{\alpha\in\mathfrak{A}}(
\mathcal{X}_{\alpha}\lor\mathcal{Z})$. For the reverse inclusion we
may assume $\mathfrak{A}=\{1,2\}$. Let $F\in(\mathcal{X}_{1}\lor
\mathcal{Z})\land(\mathcal{X}_{2}\lor\mathcal{Z})$. Then a.s.
$\mathbbm{1}_{F}=\mathbb{P}[F\vert\mathcal{X}_{1}\lor\mathcal{Z}]$
(because $F\in\mathcal{X}_{1}\lor\mathcal{Z}$). Let us now show that
if $F\in\mathcal{X}_{2}\lor\mathcal{Z}$, then $\mathbb{P}[F\vert
\mathcal{X}_{1}\lor\mathcal{Z}]\in\mathcal{Z}/\mathcal
{B}_{[-\infty
,\infty]}$; this will conclude the argument. Take $X_{2}\in
\mathcal{X}_{2}$ and $Z\in\mathcal{Z}$. Then a.s. $\mathbb{P}[X_{2}
\cap Z\vert\mathcal{X}_{1}\lor\mathcal{Z}]=\mathbbm{1}_{Z}
\mathbb{P}[X_{2}\vert\mathcal{X}_{1}\lor\mathcal{Z}]$. Thus by a
$\pi/\lambda$-argument it will suffice to establish that $
\mathbb{P}[X_{2}\vert\mathcal{X}_{1}\lor\mathcal{Z}]\in\mathcal{Z}/
\mathcal{B}_{[-\infty,\infty]}$. For this, just argue that a.s.
$\mathbb{P}[X_{2}\vert\mathcal{X}_{1}\lor\mathcal{Z}]=\mathbb{P}[X
_{2}\vert\mathcal{Z}]$: let $X_{1}\in\mathcal{X}_{1}$ and
$Z\in\mathcal{Z}$; then $\mathbb{P}(X_{2}\cap X_{1}\cap Z)=
\mathbb{E}[\mathbb{P}[X_{2}\vert\mathcal{Z}];X_{1}\cap Z]$ because
$\mathcal{X}_{1}$ is conditionally independent\index{conditional independence} of $\mathcal{X}_{2}$
given $\mathcal{Z}$; another $\pi/\lambda$-argument allows to
conclude.
\end{proof}

A further substantial statement involving conditional independence\index{conditional independence} and
distributivity\index{distributivity} is the following. It generalizes
Proposition~\ref{distributivity} in the case when $\mathfrak{B}$ is a
two-point set.
%
%p3.24 #&#
%
\begin{proposition}[Distributivity\index{distributivity} V]
\label{prop:distV}
Let $(\mathcal{X}_{\alpha i})_{(\alpha,i)\in\mathfrak{A}\times\{1,2
\}}$ be a family in $\varLambda$, $\mathfrak{A}$ non-empty. Set
$\mathcal{X}_{i}:=\lor_{\alpha\in\mathfrak{A}} \mathcal
{X}_{\alpha
i}$ for $i\in\{1,2\}$. Assume that for each finite non-empty
$A\subset\mathfrak{A}$, $\mathcal{X}_{1}$ is conditionally independent\index{conditional independence}
of $\mathcal{X}_{2}$ given $\land_{\alpha\in A}\mathcal{X}_{\alpha1}$
and also given $\land_{\alpha\in A}\mathcal{X}_{\alpha2}$. Then
%
%e3.2 #&#
%
\begin{equation}
\label{eq:dist-V} \land_{\alpha\in\mathfrak{A}}(\mathcal{X}_{\alpha1}\lor
\mathcal{X}_{\alpha2})=(\land_{\alpha\in\mathfrak{A}}\mathcal{X}
_{\alpha1})\lor(\land_{\alpha\in\mathfrak{A}}\mathcal{X}_{\alpha
2}) .
\end{equation}
\end{proposition}
\begin{proof}
By decreasing martingale convergence,\index{martingale convergence} $\mathcal{X}_{1}$ is conditionally
independent\index{conditional independence} of $\mathcal{X}_{2}$ given $\land_{\alpha\in
\mathfrak{A}}\mathcal{X}_{\alpha1}$ and also given $
\land_{\alpha\in\mathfrak{A}}\mathcal{X}_{\alpha2}$. Therefore, by
the same reduction as at the start of the proof of
Proposition~\ref{distributivity}, it suffices to establish the claim in
the following two cases.
\begin{enumerate}[(A)]%
\item
\label{V:A}
$\mathcal{X}_{\alpha2}=\mathcal{X}_{2}$ for all $\alpha\in
\mathfrak{A}$.
\item
\label{V:B}
$\mathcal{A}=\{1,2\}$, $\mathcal{X}_{11}\subset\mathcal{X}_{21}$,
$\mathcal{X}_{22}\subset\mathcal{X}_{12}$.
\end{enumerate}

\ref{V:A}. Suppose \eqref{eq:dist-V} has been established for
$\mathfrak{A}$ finite (all the time assuming \ref{V:A}, of course).\vadjust{\goodbreak} Let
$A\subset\mathfrak{A}$ be finite and non-empty and $(X_{1},X_{2})
\in\mathcal{X}_{1}\times\mathcal{X}_{2}$. Then, because $
\mathcal{X}_{1} \perp\!\!\!\!
\perp_{\land_{\alpha\in A}\mathcal{X}_{\alpha1}} \mathcal{X}_{2}$,
a.s. $\mathbb{P}[X_{1}\cap X_{2}\vert(\land_{\alpha\in A}
\mathcal{X}_{\alpha1})\lor\mathcal{X}_{2}]=\mathbb{P}[X_{1}\vert
\land_{\alpha\in A}\mathcal{X}_{\alpha1}]\mathbbm{1}_{X_{2}}$. By
decreasing martingale convergence\index{martingale convergence} and the assumption made, it follows
that $\mathbb{P}[X_{1}\cap X_{2}\vert\land_{\alpha\in\mathfrak{A}}(
\mathcal{X}_{\alpha1}\lor\mathcal{X}_{2})]\in((
\land_{\alpha\in\mathfrak{A}}\mathcal{X}_{\alpha1})\lor
\mathcal{X}_{2})/\mathcal{B}_{[-\infty,\infty]}$, and we conclude as
usual. Then it remains to establish the claim for finite $\mathcal{A}$,
and by induction for $\mathcal{A}=\{1,2\}$. The remainder of the proof
is now the same as in the proof of item \ref{dist:a} of
Proposition~\ref{distributivity}, except that, as appropriate, one
appeals to conditional independence in lieu of independence.\index{conditional independence}

\ref{V:B}. This is proved just as in the final part of the proof of item
\ref{dist:b} of Proposition~\ref{distributivity} (only the final part
is relevant because here a priori $\mathcal{A}=\mathcal{B}=\{1,2\}$),
except that again one appeals to conditional independence in lieu of
independence,\index{conditional independence} as appropriate.
\end{proof}

%c3.25 #&#
%
\begin{corollary}[Distributivity VI\index{distributivity}]
\cite{lindvall1986}, \cite[Exercise 2.5(1)]{chaumont-yor}. If
$\mathcal{Y}\in\varLambda$ and a nonincreasing sequence $(\mathcal{X}
_{n})_{n\in\mathbb{N}}$ from $\varLambda$ are such that $\mathcal{Y}
\perp\!\!\!\!\perp_{\mathcal{X}_{n}}\mathcal{X}_{1}$ for all
$n\in\mathbb{N}$, then $\land_{n\in\mathbb{N}}(\mathcal{X}_{n}
\lor\mathcal{Y})=(\land_{n\in\mathbb{N}}\mathcal{X}_{n})\lor
\mathcal{Y}$.\qed
\end{corollary}

%r3.26 #&#
%
\begin{remark}
\label{corollary:VI}
The generalization to a general $\mathfrak{B}$ in lieu of $\{1,2\}$ in
Proposition~\ref{prop:distV} seems too cumbersome to be of any value,
and we omit making it explicit.
\end{remark}

Finally we return yet again to complements.\index{complements} In the following it is
investigated what happens if one is given $\mathcal{A}\perp\!\!\!\!
\perp\mathcal{B}$ from $\varLambda$, and one enlarges $\mathcal{A}$ by
an independent complement $\mathcal{X}$ to form $\mathcal{A}'=
\mathcal{A}+\mathcal{X}$, while reducing $\mathcal{B}$ to $
\mathcal{B}'$ through an independent complement $\mathcal{Y}$,
$\mathcal{B}'+\mathcal{Y}=\mathcal{B}$, in such a manner that
$\mathcal{A}'\perp\!\!\!\!\perp\mathcal{B}'$, and that between them
$\mathcal{A}'$ and $\mathcal{B}'$ generate the same $\sigma$-field as
$\mathcal{A}$ and $\mathcal{B}$ do. (We will see in
Section~\ref{section:application} why this is an interesting situation
to consider.)

%p3.27 #&#
%
\begin{proposition}[Two-sided complements]
\label{proposition:two-sided}
Let $\{\mathcal{A},\mathcal{B},\mathcal{A}',\mathcal{B}'\}\subset
\varLambda$ be such that $\mathcal{A}+\mathcal{B}=\mathcal{A}'+
\mathcal{B}'$.
\begin{enumerate}[(i)]%
\item
\label{lemma:ii}
There is at most one $\mathcal{X}\in\varLambda$ such that $\mathcal{A}+
\mathcal{X}=\mathcal{A}'$ and $\mathcal{B}'+\mathcal{X}=\mathcal{B}$,
namely $\mathcal{A}'\land\mathcal{B}$.
\item
\label{lemma:iii}
Let $\{\mathcal{X},\mathcal{Y}\}\subset\varLambda$ be such that
$\mathcal{A}+\mathcal{X}=\mathcal{A}'$ and $\mathcal{B}'+\mathcal{Y}=
\mathcal{B}$. The following statements are equivalent:
\begin{enumerate}[(a)]%
\item
\label{lemma:iii:a}
There is $\mathcal{Z}\in\varLambda$ with $\mathcal{A}+\mathcal{Z}=
\mathcal{A}'$ and $\mathcal{B}'+\mathcal{Z}=\mathcal{B}$.
\item
\label{lemma:iii:d}
$\mathcal{A}+(\mathcal{A}'\land\mathcal{B})+\mathcal{B}'=\mathcal{A}+
\mathcal{B}$ ($=\mathcal{A}'+\mathcal{B}'$).
\item
\label{lemma:iii:b}
$\mathcal{X}\subset\mathcal{A}\lor(\mathcal{A}'\land\mathcal
{B})$ and
$\mathcal{Y}\subset\mathcal{B}'\lor(\mathcal{A}'\land\mathcal{B})$.
\item
\label{lemma:iii:g}
There is $\mathcal{X}'\in\varLambda$ with $\mathcal{X}'\subset
\mathcal{B}$ and $\mathcal{A}+\mathcal{X}'=\mathcal{A}'$ and there is
 $\mathcal{Y}'\in\varLambda$ with $\mathcal{Y}'\subset\mathcal{A}'$ and
$\mathcal{B}'+\mathcal{Y}'=\mathcal{B}$.
\item
\label{lemma:iii:e}
$\mathbb{P}[B\vert\mathcal{A}']\in\mathcal{B}/\mathcal
{B}_{[-\infty
,\infty]}$ for $B\in\mathcal{B}$ and $\mathbb{P}[A'\vert
\mathcal{B}]\in\mathcal{A}'/\mathcal{B}_{[-\infty,\infty]}$ for
$A'\in\mathcal{A}'$.
\end{enumerate}
\end{enumerate}
\end{proposition}

%e3.28 #&#
%
\begin{example}
\label{example:nudge}
Let $\xi_{i}$, $i\in\{1,2,3\}$, be independent equiprobable signs.\index{independent ! equiprobable signs} Let
$\mathcal{A}:=\overline{\sigma}(\xi_{1})$, $\mathcal
{B}':=\overline{
\sigma}(\xi_{2})$, $\mathcal{X}:=\overline{\sigma}(\xi_{3})$,
$\mathcal{Y}:=\overline{\sigma}(\{\xi_{1}=\xi_{3}=1 \text{ or
}\xi
_{3}\xi_{2}=\xi_{1}=-1\})$, $\mathcal{A}':=\mathcal{A}+ \mathcal{X}$,
$\mathcal{B}:=\mathcal{B}'+ \mathcal{Y}$. It is then straightforward to
check, for instance by considering the induced partitions, that
$\mathcal{A}+\mathcal{B}= \overline{\sigma}(\xi_{1},\xi_{2},\xi_{3})=
\mathcal{A}'+\mathcal{B}'$, while $\mathcal{A}'\land\mathcal{B}
\subset0_{\varLambda}$, so that in particular $\mathcal{A}+(
\mathcal{A}'\land\mathcal{B})+\mathcal{B}'\ne\mathcal{A}+
\mathcal{B}$. This ``discrete'' example can be tweaked to a
``continuous'' one, just like it was done in
Example~\ref{example:tweak}.
\end{example}

%r3.29 #&#
%
\begin{remark}
One would call $\mathcal{X}$ satisfying the relations stipulated by
\ref{lemma:ii} a two-sided complement of $(\mathcal{A},\mathcal{B})$ in
$(\mathcal{A}',\mathcal{B}')$. Unlike the usual ``one-sided''
complement, it is always unique, if it exists. However, by
Example~\ref{example:nudge}, the ``existence of one-sided complements
on both sides'', i.e. what is the starting assumption of
\ref{lemma:iii}, does not ensure the existence of a two-sided complement
(which is \ref{lemma:iii}\ref{lemma:iii:a}).
\end{remark}
\begin{proof}
\ref{lemma:ii}. Suppose the two relations are also satisfied by a
$\mathcal{Y}\in\varLambda$ in lieu of $\mathcal{X}$. Then $\mathcal{Y}
\subset\mathcal{B}=\mathcal{B}'+\mathcal{X}$ and $\mathcal
{Y}\subset
\mathcal{A}'=\mathcal{A}+\mathcal{X}$; hence $\mathcal{Y}\subset(
\mathcal{B}'+\mathcal{X})\land(\mathcal{A}+\mathcal{X})$. But
$\mathcal{B}'$ is independent of $\mathcal{A}'$, and $\mathcal{A}'=
\mathcal{A}+\mathcal{X}$; hence $\mathcal{B}'$, $\mathcal{A}$ and
$\mathcal{X}$ are independent, so
Corollary~\ref{lemma}\ref{lemma:i} entails that $ (\mathcal{B}'+
\mathcal{X})\land(\mathcal{A}+\mathcal{X})=\mathcal{X}$. Thus
$\mathcal{Y}\subset\mathcal{X}$ and by symmetry $\mathcal{X}\subset
\mathcal{Y}$, also; hence $\mathcal{X}=\mathcal{Y}$. If $\mathcal{X}$
satisfies the relations, then they are also a fortiori satisfied by
$\mathcal{A}'\land\mathcal{B}$; by uniqueness $\mathcal{X}=
\mathcal{A}'\land\mathcal{B}$.

\ref{lemma:iii}. Suppose \ref{lemma:iii:a} holds. Then by
\ref{lemma:ii} $\mathcal{Z}=\mathcal{A}'\land\mathcal{B}$ and
\ref{lemma:iii:d}-\ref{lemma:iii:b}-\ref{lemma:iii:g} follow at once.
To see \ref{lemma:iii:e}, let $B'\in\mathcal{B}'$ and $Z\in
\mathcal{Z}$. Then a.s. $\mathbb{P}[B'\cap Z\vert\mathcal{A}']=
\mathbb{P}[B'\cap Z\vert\mathcal{A}\lor\mathcal{Z}]=\mathbbm{1}_{Z}
\mathbb{P}[B'\vert\mathcal{A}\lor\mathcal{Z}]=\mathbbm{1}_{Z}
\mathbb{P}(B')\in\mathcal{Z}/\mathcal{B}_{[-\infty,\infty
]}\subset
\mathcal{B}/\mathcal{B}_{[-\infty,\infty]}$. The general case obtains
by a $\pi/\lambda$-argument and then the second part by symmetry.
Conversely, if any of
\ref{lemma:iii:d}-\ref{lemma:iii:b}-\ref{lemma:iii:g} obtains, then it
is straightforward to check that one can take $\mathcal{Z}=
\mathcal{A}'\land\mathcal{B}$ in \ref{lemma:iii:a} (of course by
\ref{lemma:ii} there is no other choice for $\mathcal{Z}$). Finally we
verify that \ref{lemma:iii:e} implies $\mathcal{X}\subset\mathcal{A}
\lor(\mathcal{A}'\land\mathcal{B})$ (by \ref{lemma:iii:b} and symmetry
it will be enough). The assumption entails that $\mathbb{P}[B\vert
\mathcal{A}']=\mathbb{P}[B\vert\mathcal{A}'\land\mathcal{B}]$
a.s. for
$B\in\mathcal{B}$. Let $X\in\mathcal{X}$; it will be sufficient to
show that a.s. $\mathbb{P}[X\vert\mathcal{A}\lor(\mathcal{A}'\land
\mathcal{B})]=\mathbbm{1}_{X}$, and then by a $\pi/\lambda$-argument,
that $\mathbb{E}[\mathbb{P}[X\vert\mathcal{A}\lor(\mathcal{A}'
\land\mathcal{B})];A\cap B]=\mathbb{P}(X\cap A\cap B)$ for
$A\in\mathcal{A}$, $B\in\mathcal{B}$. Now because $(\mathcal{A}'
\land\mathcal{B})\lor\sigma(B)\subset\mathcal{B}\perp\!\!\!\!
\perp\mathcal{A}$, we find indeed that $\mathbb{E}[\mathbb{P}[X
\vert\mathcal{A}\lor(\mathcal{A}'\land\mathcal{B})];A\cap B]=
\mathbb{E}[\mathbb{P}[X\cap A\vert\mathcal{A}\lor(\mathcal{A}'
\land\mathcal{B})];B]=\mathbb{E}[\mathbb{P}[B\vert\mathcal{A}\lor(
\mathcal{A}'\land\mathcal{B})];X\cap A]=\mathbb{E}[\mathbb{P}[B
\vert\mathcal{A}'\land\mathcal{B}];X\cap A]=\mathbb{E}[\mathbb{P}[B
\vert\mathcal{A}'];X\cap A]=\mathbb{P}(X\cap A\cap B)$.
\end{proof}

%s4 #&#
\section{An application to the problem of innovation}%
\label{section:application}
Let $\mathcal{F}=(\mathcal{F}_{n})_{n\in\mathbb{N}}$ be a nonincreasing
sequence in $\varLambda$ and let $\mathcal{G}=(\mathcal{G}_{n})_{n
\in\mathbb{N}}$ be a nondecreasing sequence in $\varLambda$ such that
$\mathcal{F}_{n}\lor\mathcal{G}_{n}=\mathcal{F}_{1}\lor\mathcal{G}
_{1}$ for all $n\in\mathbb{N}$. Set $\mathcal{F}_{\infty}:=
\land_{n\in\mathbb{N}}\mathcal{F}_{n}$ and $\mathcal{G}_{\infty}:=
\lor_{n\in\mathbb{N}}\mathcal{G}_{n}$, as well as (for convenience)
$\mathcal{G}_{0}:=0_{\varLambda}$, $\mathcal{F}_{0}:=\mathcal{F}_{1}
\lor\mathcal{G}_{1}$. We are interested in specifying (equivalent)
conditions under which $\mathcal{F}_{\infty}\lor\mathcal{G}_{\infty
}=\mathcal{F}_{0}$. We have of course a priori the inclusion
$\mathcal{F}_{\infty}\lor\mathcal{G}_{\infty}\subset\mathcal{F}
_{0}$.

%r4.1 #&#
%
\begin{remark}
\label{remark:weiz}
Since $\mathcal{F}_{n}\lor\mathcal{G}_{\infty}=\mathcal{F}_{0}$ for
all $n\in\mathbb{N}$, the statement $\mathcal{F}_{\infty}\lor
\mathcal{G}_{\infty}=\mathcal{F}_{0}$ is equivalent to $(
\land_{n\in\mathbb{N}}\mathcal{F}_{n})\lor\mathcal{G}_{\infty}=
\land_{n\in\mathbb{N}}(\mathcal{F}_{n}\lor\mathcal{G}_{\infty
})$, and
the conditions of the theorem of \cite{weiz} apply. For instance,
assume (i) $\mathcal{F}_{0}$ is countably generated up to negligible
sets; and (ii) $\mathcal{F}_{\infty}=0_{\varLambda}$. Take a regular
version $(\mathbb{P}_{\mathcal{G}_{\infty}}^{\omega})_{\omega
\in\varOmega}$ of the conditional probability on $\mathcal{F}_{0}$ given
$\mathcal{G}_{\infty}$ [it means that $\mathcal{G}_{\infty}/
\mathcal{B}_{[0,1]}\ni\mathbb{P}_{\mathcal{G}_{\infty}}^{\cdot}(A)=
\mathbb{P}[A\vert\mathcal{G}_{\infty}]$ a.s. for all $A\in
\mathcal{F}_{0}$, and $\mathbb{P}_{\mathcal{G}_{\infty}}^{\omega}$ is
a probability measure on $\mathcal{F}_{0}$ for each $\omega\in
\varOmega$]. Then we can write Theorem.e in  \cite{weiz} as $
\mathcal{F}_{\infty}\lor\mathcal{G}_{\infty}=\mathcal{F}_{0}$ iff
$\mathbb{P}_{\mathcal{G}_{\infty}}^{\omega}$ is trivial on
$\mathcal{F}_{\infty}$ a.s. in $\omega\in\varOmega$.
\end{remark}
We will restrict our attention to the case when there are strong
independence properties. A typical example of the type of situation that
we have in mind and when the equality $\mathcal{F}_{\infty}\lor
\mathcal{G}_{\infty}=\mathcal{F}_{0}$ (nevertheless) fails was the
content of Example~\ref{example:vanishing} in the introduction.

%\smallskip
%\noindent
%\textbf{Example~\ref{example:vanishing} continued.}
\begin{exampcont}
With regard to
Remark~\ref{remark:weiz}, note that (in the context of
Example~\ref{example:vanishing}) $\mathcal{G}_{\infty}=\overline{
\sigma}(\{A\in\mathcal{M}:A=-A\})$. Indeed one checks easily that\vadjust{\eject}
$\sigma(\xi_{1}\xi_{2},\break \xi_{2}\xi_{3},\ldots)\subset\{A\in
\mathcal{M}:A=-A\}$. Conversely, if for a $C\in(2^{\{-1,1\}})^{
\otimes\mathbb{N}}$, $A=(\xi_{1},\break \xi_{1}\xi_{2},\xi_{2}\xi_{3},
\ldots)^{-1}(C)=-A$, then $A=(\xi_{1},\xi_{1}\xi_{2},\xi_{2}\xi_{3},
\ldots)^{-1}(C)=(-\xi_{1},\xi_{1}\xi_{2},\break \xi_{2}\xi_{3},\ldots)^{-1}(C)=(
\xi_{1}\xi_{2},\xi_{2}\xi_{3}, \ldots)^{-1}(\mathrm
{pr}_{2,3,\ldots}(C))$;
as a consequence, Blackwell's theorem
\cite[Theorem~III.17]{meyer} shows that $A\in\sigma(\xi_{1}\xi_{2},
\xi_{2}\xi_{3},\ldots)$, so that also $\sigma(\xi_{1}\xi_{2},\break  \xi
_{2}\xi_{3},\ldots)\supset\{A\in\mathcal{M}:A=-A\}$. Thus in
Remark~\ref{remark:weiz} we may take $\mathbb{E}_{\mathbb{P}^{\cdot}
_{\mathcal{G}_{\infty}}}[f]=(f+f\circ(-\mathrm{id}_{\varOmega}))/2$ for
$f\in((2^{\{-1,1\}})^{\otimes\mathbb{N}})/\mathcal{B}_{[0,\infty]}$.
For this choice $\mathbb{P}^{\omega}_{\mathcal{G}_{\infty}}$ is
nontrivial on $\mathcal{F}_{\infty}$ for arbitrary $\omega\in\varOmega$ (take,
e.g., $f$ equal to the indicator of the event $A_{\omega}:=\{\xi_{n}=
\omega(n)\text{ for all sufficiently large }n\in\mathbb{N}\}$).
\end{exampcont}

Here is now a general result that motivates the investigation of
two-sided complements in Proposition~\ref{proposition:two-sided}.

%p4.2 #&#
%
\begin{proposition}
Let $\mathcal{H}=(\mathcal{H}_{n})_{n\in
\mathbb{N}}$ be a sequence in $\varLambda$ such that $\mathcal{F}_{n}=\mathcal{F}
_{n+1}+ \mathcal{H}_{n+1}$ and $\mathcal{G}_{n+1}=\mathcal{G}_{n}+
\mathcal{H}_{n+1}$ for all $n\in\mathbb{N}_{0}$. (One would say that the sequence $\mathcal{H}$ ``innovates'' $(\mathcal{F},\mathcal{G})$.) Then $\mathcal{H}_{n}=\mathcal{G}_{n}\land\mathcal{F}_{n-1}$ for
all $n\in\mathbb{N}$, and the following statements are equivalent.
\begin{enumerate}[(i)]%
\item
\label{innovation:0}
$\mathcal{F}_{\infty}\lor\mathcal{G}_{\infty}=\mathcal{F}_{0}$.
\item
\label{innovation:i}
$\mathcal{F}_{n}=\mathcal{F}_{\infty}\lor[\lor_{k\in\mathbb{N}_{>n}}
\mathcal{H}_{k}]$ for all $n\in\mathbb{N}_{0}$.
\item
\label{innovation:ii}
$\mathcal{F}_{n}=\mathcal{F}_{\infty}\lor[\lor_{k\in\mathbb
{N}_{> n}}
\mathcal{H}_{k}]$ for some $n\in\mathbb{N}_{0}$.
\end{enumerate}
\end{proposition}
\begin{proof}
We have $\mathcal{F}_{n}+\mathcal{G}_{n}=
\mathcal{F}_{n+1}+\mathcal{G}_{n+1}$ for all $n\in\mathbb{N}_{0}$. Now the expressions for the $\mathcal{H}_{n}$,
$n\in\mathbb{N}$, follow from
Proposition~\ref{proposition:two-sided}\ref{lemma:ii}. Note also that
$\mathcal{G}_{n}=\mathcal{H}_{1}\lor\cdots\lor\mathcal{H}_{n}$ for
all $n\in\mathbb{N}_{0}$.

The implication \ref{innovation:i} $\Rightarrow$ \ref{innovation:ii}
is trivial.

\ref{innovation:0} $\Rightarrow$ \ref{innovation:i}. The inclusion
$\supset$ is clear. Conversely, if $F \in\mathcal{F}_{n}$, then a.s.
$\mathbbm{1}_{F}=\mathbb{P}[F\vert\mathcal{F}_{0}]=\mathbb{P}[F
\vert\mathcal{F}_{\infty}\lor\mathcal{G}_{\infty}]=\mathbb{P}[F
\vert\mathcal{F}_{\infty}\lor\mathcal{G}_{n}\lor[
\lor_{k\in\mathbb{N}_{>n}}\mathcal{H}_{k}]]=\mathbb{P}[F\vert
\mathcal{F}_{\infty}\lor\break [\lor_{k\in\mathbb{N}_{>n}}\mathcal{H}_{k}]]$,
since $\mathcal{G}_{n}\perp\!\!\!\!\perp\mathcal{F}_{n}\supset
\sigma(F)\lor\mathcal{F}_{\infty}\lor[\lor_{k\in\mathbb{N}_{>n}}
\mathcal{H}_{k}]$.

\ref{innovation:ii} $\Rightarrow$ \ref{innovation:0}. $\mathcal{F}
_{\infty}\lor\mathcal{G}_{\infty}=\mathcal{F}_{\infty}\lor
\mathcal{G}_{n}\lor[\lor_{k\in\mathbb{N}_{>n}}\mathcal{H}_{k}]=
\mathcal{F}_{n}\lor\mathcal{G}_{n}=\mathcal{F}_{0}$.
\end{proof}

%\begin{appendix}
%\end{appendix}

\begin{funding}
Financial support from the \gsponsor[id=GS1,sponsor-id=501100004329]{Slovenian Research Agency} is acknowledged (programme No. \gnumber[refid=GS1]{P1-0222}).\end{funding}

\begin{acknowledgement}
The author is grateful to an anonymous Referee
for providing guidance that helped to improve the presentation of this
paper.
%[title={Acknowledgments}]
\end{acknowledgement}


\begin{thebibliography}{99}

%b1 ###
\bibitem{burkholder}
%
\begin{barticle}
\bauthor{\bsnm{Burkholder}, \binits{D.L.}},
\bauthor{\bsnm{Chow}, \binits{Y.S.}}:
\batitle{Iterates of conditional expectation operators}.
\bjtitle{Proc. Am. Math. Soc.}
\bvolume{12}(\bissue{3}),
\bfpage{490}--\blpage{495}
(\byear{1961}).
\bid{doi={10.2307/2034224}, mr={0142144}}
\end{barticle}
%
%
\OrigBibText
%
\begin{barticle}
\bauthor{\bsnm{Burkholder}, \binits{D.L.}},
\bauthor{\bsnm{Chow}, \binits{Y.S.}}:
\batitle{Iterates of conditional expectation operators}.
\bjtitle{Proceedings of the American Mathematical Society}
\bvolume{12}(\bissue{3}),
\bfpage{490}--\blpage{495}
(\byear{1961})
\end{barticle}
%
\endOrigBibText
\bptok{structpyb}%
\endbibitem

%b2 ###
\bibitem{chaumont-yor}
%
\begin{bbook}
\bauthor{\bsnm{Chaumont}, \binits{L.}},
\bauthor{\bsnm{Yor}, \binits{M.}}:
\bbtitle{Exercises in Probability: A Guided Tour from Measure Theory to Random Processes, Via Conditioning}.
\bsertitle{Cambridge Series in Statistical and Probabilistic Mathematics}.
\bpublisher{Cambridge University Press}
(\byear{2003}).
\bid{doi={10.1017/\\CBO9780511610813}, mr={2016344}}
\end{bbook}
%
%
\OrigBibText
%
\begin{bbook}
\bauthor{\bsnm{Chaumont}, \binits{L.}},
\bauthor{\bsnm{Yor}, \binits{M.}}:
\bbtitle{Exercises in Probability: A Guided Tour from Measure Theory
to Random
Processes, Via Conditioning}.
\bsertitle{Cambridge Series in Statistical and Probabilistic Mathematics}.
\bpublisher{Cambridge University Press}
(\byear{2003})
\end{bbook}
%
\endOrigBibText
\bptok{structpyb}%
\endbibitem\goodbreak

%b3 ###
\bibitem{emery}
%
\begin{bchapter}
\bauthor{\bsnm{{\'E}mery}, \binits{M.}},
\bauthor{\bsnm{Schachermayer}, \binits{W.}}:
\bctitle{A remark on {T}sirelson's stochastic differential equation}.
In: \beditor{\bsnm{Az{\'e}ma}, \binits{J.}},
\beditor{\bsnm{{\'E}mery}, \binits{M.}},
\beditor{\bsnm{Ledoux}, \binits{M.}},
\beditor{\bsnm{Yor}, \binits{M.}} (eds.)
\bbtitle{S{\'e}minaire de Probabilit{\'e}s XXXIII},
pp.~\bfpage{291}--\blpage{303}.
\bpublisher{Springer},
\blocation{Berlin, Heidelberg}
(\byear{1999}).
\bid{doi={10.1007/\\BFb0096508}, mr={1768019}}
\end{bchapter}
%
%
\OrigBibText
%
\begin{bchapter}
\bauthor{\bsnm{{\'E}mery}, \binits{M.}},
\bauthor{\bsnm{Schachermayer}, \binits{W.}}:
\bctitle{A remark on {T}sirelson's stochastic differential equation}.
In: \beditor{\bsnm{Az{\'e}ma}, \binits{J.}},
\beditor{\bsnm{{\'E}mery}, \binits{M.}},
\beditor{\bsnm{Ledoux}, \binits{M.}},
\beditor{\bsnm{Yor}, \binits{M.}} (eds.)
\bbtitle{S{\'e}minaire de Probabilit{\'e}s XXXIII},
pp.~\bfpage{291}--\blpage{303}.
\bpublisher{Springer},
\blocation{Berlin, Heidelberg}
(\byear{1999})
\end{bchapter}
%
\endOrigBibText
\bptok{structpyb}%
\endbibitem

%b4 ###
\bibitem{emery-sch}
%
\begin{bchapter}
\bauthor{\bsnm{{\'E}mery}, \binits{M.}},
\bauthor{\bsnm{Schachermayer}, \binits{W.}}:
\bctitle{On {V}ershik's standardness criterion and {T}sirelson's notion of cosiness}.
In: \beditor{\bsnm{Az{\'e}ma}, \binits{J.}},
\beditor{\bsnm{{\'E}mery}, \binits{M.}},
\beditor{\bsnm{Ledoux}, \binits{M.}},
\beditor{\bsnm{Yor}, \binits{M.}} (eds.)
\bbtitle{S{\'e}minaire de Probabilit{\'e}s XXXV},
pp.~\bfpage{265}--\blpage{305}.
\bpublisher{Springer},
\blocation{Berlin, Heidelberg}
(\byear{2001}).
\bid{doi={10.1007/978-3-540-44671-2\_20}, mr={1837293}}
\end{bchapter}
%
%
\OrigBibText
%
\begin{bchapter}
\bauthor{\bsnm{{\'E}mery}, \binits{M.}},
\bauthor{\bsnm{Schachermayer}, \binits{W.}}:
\bctitle{On {V}ershik's standardness criterion and {T}sirelson's
notion of
cosiness}.
In: \beditor{\bsnm{Az{\'e}ma}, \binits{J.}},
\beditor{\bsnm{{\'E}mery}, \binits{M.}},
\beditor{\bsnm{Ledoux}, \binits{M.}},
\beditor{\bsnm{Yor}, \binits{M.}} (eds.)
\bbtitle{S{\'e}minaire de Probabilit{\'e}s XXXV},
pp.~\bfpage{265}--\blpage{305}.
\bpublisher{Springer},
\blocation{Berlin, Heidelberg}
(\byear{2001})
\end{bchapter}
%
\endOrigBibText
\bptok{structpyb}%
\endbibitem

%b5 ###
\bibitem{feldman}
%
\begin{barticle}
\bauthor{\bsnm{Feldman}, \binits{J.}}:
\batitle{Decomposable processes and continuous products of probability spaces}.
\bjtitle{J. Funct. Anal.}
\bvolume{8}(\bissue{1}),
\bfpage{1}--\blpage{51}
(\byear{1971}).
\bid{doi={10.1016/0022-\\1236(71)90017-6}, mr={0290436}}
\end{barticle}
%
%
\OrigBibText
%
\begin{barticle}
\bauthor{\bsnm{Feldman}, \binits{J.}}:
\batitle{Decomposable processes and continuous products of probability spaces}.
\bjtitle{Journal of Functional Analysis}
\bvolume{8}(\bissue{1}),
\bfpage{1}--\blpage{51}
(\byear{1971})
\end{barticle}
%
\endOrigBibText
\bptok{structpyb}%
\endbibitem

%b6 ###
\bibitem{kallenberg}
%
\begin{bbook}
\bauthor{\bsnm{Kallenberg}, \binits{O.}}:
\bbtitle{{Foundations of Modern Probability}}.
\bsertitle{Probability and Its Applications}.
\bpublisher{Springer},
\blocation{New York Berlin Heidelberg}
(\byear{1997}).
\bid{mr={1464694}}
\end{bbook}
%
%
\OrigBibText
%
\begin{bbook}
\bauthor{\bsnm{Kallenberg}, \binits{O.}}:
\bbtitle{{Foundations of Modern Probability}}.
\bsertitle{Probability and Its Applications}.
\bpublisher{Springer},
\blocation{New York Berlin Heidelberg}
(\byear{1997})
\end{bbook}
%
\endOrigBibText
\bptok{structpyb}%
\endbibitem

%b7 ###
\bibitem{lindvall1986}
%
\begin{barticle}
\bauthor{\bsnm{Lindvall}, \binits{T.}},
\bauthor{\bsnm{Rogers}, \binits{L.C.G.}}:
\batitle{Coupling of multidimensional diffusions by reflection}.
\bjtitle{Ann. Probab.}
\bvolume{14}(\bissue{3}),
\bfpage{860}--\blpage{872}
(\byear{1986}).
\bid{mr={0841588}}
\end{barticle}
%
%
\OrigBibText
%
\begin{barticle}
\bauthor{\bsnm{Lindvall}, \binits{T.}},
\bauthor{\bsnm{Rogers}, \binits{L.C.G.}}:
\batitle{Coupling of multidimensional diffusions by reflection}.
\bjtitle{The Annals of Probability}
\bvolume{14}(\bissue{3}),
\bfpage{860}--\blpage{872}
(\byear{1986})
\end{barticle}
%
\endOrigBibText
\bptok{structpyb}%
\endbibitem

%b8 ###
\bibitem{meyer}
%
\begin{bbook}
\bauthor{\bsnm{Meyer}, \binits{P.A.}}:
\bbtitle{Probability and Potentials}.
\bsertitle{Blaisdell book in pure and applied mathematics}.
\bpublisher{Blaisdell Publishing Company}
(\byear{1966}).
\bid{mr={0205288}}
\end{bbook}
%
%
\OrigBibText
%
\begin{bbook}
\bauthor{\bsnm{Meyer}, \binits{P.A.}}:
\bbtitle{Probability and Potentials}.
\bsertitle{Blaisdell book in pure and applied mathematics}.
\bpublisher{Blaisdell Publishing Company}
(\byear{1966})
\end{bbook}
%
\endOrigBibText
\bptok{structpyb}%
\endbibitem

%b9 ###
\bibitem{neveu}
%
\begin{bbook}
\bauthor{\bsnm{Neveu}, \binits{J.}},
\bauthor{\bsnm{Speed}, \binits{T.P.}}:
\bbtitle{Discrete-parameter Martingales}.
\bsertitle{North-Holland mathematical library}.
\bpublisher{North-Holland}
(\byear{1975}).
\bid{mr={0402915}}
\end{bbook}
%
%
\OrigBibText
%
\begin{bbook}
\bauthor{\bsnm{Neveu}, \binits{J.}},
\bauthor{\bsnm{Speed}, \binits{T.P.}}:
\bbtitle{Discrete-parameter Martingales}.
\bsertitle{North-Holland mathematical library}.
\bpublisher{North-Holland}
(\byear{1975})
\end{bbook}
%
\endOrigBibText
\bptok{structpyb}%
\endbibitem

%b10 ###
\bibitem{revuz-yor}
%
\begin{bbook}
\bauthor{\bsnm{Revuz}, \binits{D.}},
\bauthor{\bsnm{Yor}, \binits{M.}}:
\bbtitle{Continuous Martingales and Brownian Motion}.
\bsertitle{Grundlehren der mathematischen Wissenschaften}.
\bpublisher{Springer}
(\byear{2004}).
\bid{doi={10.1007/\\978-3-662-21726-9}, mr={1083357}}
\end{bbook}
%
%
\OrigBibText
%
\begin{bbook}
\bauthor{\bsnm{Revuz}, \binits{D.}},
\bauthor{\bsnm{Yor}, \binits{M.}}:
\bbtitle{Continuous Martingales and Brownian Motion}.
\bsertitle{Grundlehren der mathematischen Wissenschaften}.
\bpublisher{Springer}
(\byear{2004})
\end{bbook}
%
\endOrigBibText
\bptok{structpyb}%
\endbibitem

%b11 ###
\bibitem{schramm}
%
\begin{barticle}
\bauthor{\bsnm{Schramm}, \binits{O.}},
\bauthor{\bsnm{Smirnov}, \binits{S.}},
\bauthor{\bsnm{Garban}, \binits{C.}}:
\batitle{On the scaling limits of planar percolation}.
\bjtitle{Ann. Probab.}
\bvolume{39}(\bissue{5}),
\bfpage{1768}--\blpage{1814}
(\byear{2011}).
\bid{doi={10.1214/11-AOP659}, mr={2884873}}
\end{barticle}
%
%
\OrigBibText
%
\begin{barticle}
\bauthor{\bsnm{Schramm}, \binits{O.}},
\bauthor{\bsnm{Smirnov}, \binits{S.}},
\bauthor{\bsnm{Garban}, \binits{C.}}:
\batitle{On the scaling limits of planar percolation}.
\bjtitle{The Annals of Probability}
\bvolume{39}(\bissue{5}),
\bfpage{1768}--\blpage{1814}
(\byear{2011})
\end{barticle}
%
\endOrigBibText
\bptok{structpyb}%
\endbibitem

%b12 ###
\bibitem{tsirelson}
%
\begin{barticle}
\bauthor{\bsnm{Tsirelson}, \binits{B.}}:
\batitle{Noise as a {B}oolean algebra of $\sigma$-fields}.
\bjtitle{Ann. Probab.}
\bvolume{42}(\bissue{1}),
\bfpage{311}--\blpage{353}
(\byear{2014}).
\bid{doi={10.1214/13-AOP861}, mr={3161487}}
\end{barticle}
%
%
\OrigBibText
%
\begin{barticle}
\bauthor{\bsnm{Tsirelson}, \binits{B.}}:
\batitle{Noise as a {B}oolean algebra of $\sigma$-fields}.
\bjtitle{The Annals of Probability}
\bvolume{42}(\bissue{1}),
\bfpage{311}--\blpage{353}
(\byear{2014})
\end{barticle}
%
\endOrigBibText
\bptok{structpyb}%
\endbibitem

%b13 ###
\bibitem{weiz}
%
\begin{barticle}
\bauthor{\bsnm{Weizs\"acker}, \binits{H.V.}}:
\batitle{Exchanging the order of taking suprema and countable intersections of $\sigma$-algebras}.
\bjtitle{Ann. I.H.P. Probab. Stat.}
\bvolume{19}(\bissue{1}),
\bfpage{91}--\blpage{100}
(\byear{1983}).
\bid{mr={0699981}}
\end{barticle}
%
%
\OrigBibText
%
\begin{barticle}
\bauthor{\bsnm{Weizs\"acker}, \binits{H.V.}}:
\batitle{Exchanging the order of taking suprema and countable intersections of $\sigma$-algebras}.
\bjtitle{Annales de l'I.H.P. Probabilit\'es et statistiques}
\bvolume{19}(\bissue{1}),
\bfpage{91}--\blpage{100}
(\byear{1983})
\end{barticle}
%
\endOrigBibText
\bptok{structpyb}%
\endbibitem

%b14 ###
\bibitem{williams}
%
\begin{bbook}
\bauthor{\bsnm{Williams}, \binits{D.}}:
\bbtitle{Probability with Martingales}.
\bsertitle{Cambridge mathematical textbooks}.
\bpublisher{Cambridge University Press}
(\byear{1991}).
\bid{doi={10.1017/\\CBO9780511813658}, mr={1155402}}
\end{bbook}
%
%
\OrigBibText
%
\begin{bbook}
\bauthor{\bsnm{Williams}, \binits{D.}}:
\bbtitle{Probability with Martingales}.
\bsertitle{Cambridge mathematical textbooks}.
\bpublisher{Cambridge University Press}
(\byear{1991})
\end{bbook}
%
\endOrigBibText
\bptok{structpyb}%
\endbibitem

%b15 ###
\bibitem{Yor1992}
%
\begin{barticle}
\bauthor{\bsnm{Yor}, \binits{M.}}:
\batitle{Tsire{l's}on's equation in discrete time}.
\bjtitle{Probab. Theory Relat. Fields}
\bvolume{91}(\bissue{2}),
\bfpage{135}--\blpage{152}
(\byear{1992}).
\bid{doi={10.1007/BF01291422}, mr={1147613}}
\end{barticle}
%
%
\OrigBibText
%
\begin{barticle}
\bauthor{\bsnm{Yor}, \binits{M.}}:
\batitle{Tsire{l's}on's equation in discrete time}.
\bjtitle{Probability Theory and Related Fields}
\bvolume{91}(\bissue{2}),
\bfpage{135}--\blpage{152}
(\byear{1992})
\end{barticle}
%
\endOrigBibText
\bptok{structpyb}%
\endbibitem

\end{thebibliography}
\end{document}